\documentclass[hidelinks,onefignum,onetabnum]{siamart251216}

\usepackage{lipsum}
\usepackage{amsfonts}
\usepackage{graphicx}
\usepackage{epstopdf}
\usepackage{algorithmic}
\usepackage{amsmath}
\usepackage{enumitem}
\usepackage{extarrows}
\usepackage{xcolor}
\usepackage{hyperref}
\usepackage{tabularx}
\usepackage{multirow}
\usepackage{fancyhdr}
\usepackage[mathscr]{euscript}
\ifpdf
  \DeclareGraphicsExtensions{.eps,.pdf,.png,.jpg}
\else
  \DeclareGraphicsExtensions{.eps}
\fi

\newcommand{\R}{\mathbb{R}}

\newcommand{\ip}[2]{\left(#1,\,#2\right)}
\newcommand{\dual}[2]{\left\langle #1,\,#2\right\rangle}
\newcommand{\norm}[1]{\left\|#1\right\|}

\newcommand{\divg}{\operatorname{div}}

\newsiamremark{remark}{Remark}
\newsiamremark{hypothesis}{Hypothesis}
\crefname{hypothesis}{Hypothesis}{Hypotheses}
\newsiamthm{claim}{Claim}
\newsiamremark{fact}{Fact}
\crefname{fact}{Fact}{Facts}
\newsiamthm{example}{Example}

\definecolor{yesgreen}{HTML}{2E7D32}
\definecolor{nored}{HTML}{C62828}
\definecolor{maybeorange}{HTML}{E65100}

\headers{Well-Posedness of Filtering Equations}{Z. Sun, S. Zhou and S. S.-T. Yau}

\title{Well-posedness of Filtering Equations in Weighted Sobolev Spaces with Unbounded System Coefficients \thanks{Submitted to the editors DATE. 
\funding{This work is supported by the National Natural Science Foundation of China (NSFC) under Grant No. 123B2020 for Zeju Sun, and the National Natural Science Foundation of China under Grant No. 42450242, and Tsinghua University Education Foundation for Stephen S.-T. Yau.}}}

\author{
	Zeju Sun
	\thanks{
		Beijing Institute of Mathematical Sciences and Applications, Beijing, China
		(\email{sunzeju@bimsa.cn}).}
	\and
	Songlin Zhou
	\thanks{
		Qiuzhen College, Tsinghua University, Beijing, China
		(\email{zhousl24@mails.tsinghua.edu.cn}).
	}
	\and
	Stephen S.-T. Yau
	\thanks{
		Department of Mathematical Sciences, Tsinghua University, Beijing, China
		(\email{yau@uic.edu}).
	}
}

\usepackage{amsopn}

\ifpdf
\hypersetup{
  pdftitle={Well-Posedness of Filtering Equations with Unbounded System Coefficients},
  pdfauthor={Zeju Sun, Songlin Zhou and Stephen S.-T. Yau}
}
\fi

\begin{document}

\maketitle

\footnotetext[5]{Zeju Sun and Songlin Zhou contributed equally to this work.}

% REQUIRED
\begin{abstract}
Nonlinear filtering problem is one of the core subjects in modern control theory. In this paper, we will study the well-posedness of the three fundamental evolution equations arising in continuous-time nonlinear filtering--the robust Duncan-Mortensen-Zakai (DMZ) equation, the stochastic DMZ equation, and the Kushner-Stratonovich equation--within a unified buffered weighted formulation. An exponential-type weight function and the corresponding weighted Sobolev spaces are introduced to enable a variational treatment of the filtering equations in a more general setting, in which the coefficients of the filtering system may be unbounded with polynomial growth. Under mild and easily verifiable assumptions, we first establish the well-posedness of the weak solution to the robust DMZ equation in these weighted spaces. Using the gauge (exponential) transformation and its inverse, these results are then transferred to the stochastic DMZ equation and the Kushner-Stratonovich equation, whose solutions are shown to exist and be unique in buffered weighted Sobolev spaces, yielding a unified treatment of all three filtering equations. Sufficient conditions for the well-posedness are also summarized, which illustrate the wide applicability of the proposed framework to general nonlinear filtering systems.
\end{abstract}

% REQUIRED

% REQUIRED

	\section{Introduction}\label{sec:intro}
%=====================================================================

The continuous-time nonlinear filtering problem concerns the model
\begin{equation}\label{eq:model}
	\left\{\begin{aligned}
		dX_t &= f(X_t)\,dt + dV_t, \qquad X_0\sim\pi_0,\\
		dY_t &= h(X_t)\,dt + dW_t, \qquad Y_0=0,
	\end{aligned}\right.\ t\in [0,T],
\end{equation}
where $T> 0$ is a fixed finite terminal time; the signal $X$ takes values in $\R^n$; the observation $Y$ takes values in $\R^m$; $V$ and $W$ are independent standard Brownian motions; $f:\mathbb{R}^{n}\rightarrow\mathbb{R}^{n}$ and $h:\mathbb{R}^{n}\rightarrow\mathbb{R}^{m}$ are transition and observation functions of certain classes; and the initial distribution $\pi_0$ is independent of $(V,W)$. The infinitesimal generator of the state process $X_{t}$, which is a second-order elliptic operator, is denoted by
\begin{equation}
	L = \frac{1}{2}\Delta + f\cdot\nabla  = \frac{1}{2}\sum_{i=1}^{n}\frac{\partial^{2}}{\partial x_{i}^{2}} + \sum_{i=1}^{n} f_{i}\frac{\partial }{\partial x_{i}}.
\end{equation}

The objective of the filtering problem is to provide accurate and real-time estimates of an unknown stochastic dynamics (namely, the state process $X_t$ in \eqref{eq:model}) based on noisy observations (namely, the observation process $Y_t$ in \eqref{eq:model}). Mathematically, the object of interest is the conditional distribution $\pi_t$ of $X_t$ given the observation $\sigma$-algebra $\mathcal{Y}_t = \sigma(Y_s : s \le t)$, defined by
\begin{equation*}
	\pi_t(\cdot) = P(X_t \in \cdot \mid \mathcal{Y}_t).
\end{equation*}

Since the notable Kalman-Bucy filter \cite{Kalman60,KalmanBucy61} was introduced in the 1960s for linear Gaussian systems, filtering theory and algorithms have found tremendous applications across a wide range of practical scenarios, including the aerospace industry \cite{Hutchinson84,Kerr87}, communication technology \cite{GordonSalmondSmith93}, finance \cite{KimNelson1999}, geoscience \cite{Evensen92}, autonomous driving \cite{Geiger2013}, robotics \cite{RenganathanSummers22}, and so on. Except for several special cases, such as linear Gaussian systems, the conditional distribution $\pi_t$ does not admit a finite-dimensional sufficient statistic. Instead, its evolution is described by several stochastic partial differential equations.

Classically, the evolution of the conditional distribution $\pi_{t}$ is governed by the Kushner-Stratonovich equation \cite{Kushner67,Stratonovich60}:
\begin{equation}\label{eq:kse}
	\pi_t(\varphi)=\pi_0(\varphi)+\int_0^t\pi_s(L\varphi)ds+
	\int_0^t[\pi_s(\varphi h)-\pi_s(\varphi)\pi_s(h)]^{\top}[dY_s-\pi_s(h)ds], \ t\in [0,T],
\end{equation}
where $\varphi\in C_{c}^{\infty}(\mathbb{R}^{n})$ is an arbitrary smooth function with compact support. 

If the conditional distribution $\pi_t$ is absolutely continuous with respect to the Lebesgue measure on $\mathbb{R}^{n}$ for each $t\in [0,T]$ and almost surely for each observation trajectory $\{Y_{t}:0\leq t\leq T\}$, then formally, its unnormalized density function $\sigma(t,x)$ satisfies the Duncan-Mortensen-Zakai (DMZ) equation \cite{Duncan67,Mortensen66,Zakai69}:
\begin{equation}
	d\sigma(t,x) = L^{*}\sigma(t,x)dt + h^{\top}(x)\sigma(t,x) dY_{t},\ t\in[0,T],
	\label{eq:dmz}
\end{equation}
which is a linear stochastic partial differential equation driven by the observation process, with $L^{*}$ the adjoint operator of $L$; see \cite{BainCrisan09,Kallianpur80,Pardoux91} for systematic accounts.

A classical idea, which goes back to Clark \cite{Clark78} and Davis \cite{Davis80}, is to remove the stochastic integral from the Zakai equation by the multiplicative gauge (exponential) transformation
\begin{equation}\label{eq:gauge}
	u(t,x) \;=\; e^{-Y_t^{\top} h(x)}\,\sigma(t,x).
\end{equation}
The transformed density $u$ satisfies, for each fixed observation path, a \emph{deterministic} linear parabolic equation with coefficients depending on the path $Y$ only through its current value $Y_t$: 
\begin{equation}\label{eq:rdmz}
	\partial_{t} u\;=\; \tfrac12\Delta u + b_Y(t,x)\cdot\nabla u + P_Y(t,x)\,u,
	\qquad u(0,\cdot)=u_0,
\end{equation}
with
\begin{equation}\label{eq:coeffs}
	\begin{aligned}
		b_Y(t,x) &= -f(x) + (\nabla h(x))^{\top}Y_t,\\
		P_Y(t,x) &= -\divg f(x) - \tfrac12|h(x)|^2 + \tfrac12\,Y_t^{\top}\Delta h(x) - Y_{t}^{\top}\nabla h(x) f(x)+ \tfrac12\big|(\nabla h(x))^{\top}Y_t\big|^2 .
	\end{aligned}
\end{equation}
Equation \eqref{eq:rdmz} is now referred to as the \emph{robust} or \emph{pathwise} DMZ equation, which underlies the robust filtering approximations in the sense of Clark and Davis and the real-time DMZ program of Yau and Yau \cite{Clark78,Davis80,YauYau00,YauYau08,LuoYau13}.  

When $f$ and $h$ are bounded with bounded derivatives, the well-posedness theory for the robust DMZ equation and for the Zakai equation is classical \cite{Pardoux79,KrylovRozovskii81,RozovskyLototsky18}.  The genuinely difficult--and practically relevant--regime is that with \emph{unbounded} coefficients: linear and polynomial drifts and sensors are the rule rather than the exception in applications, including Kalman-Bucy filtering, polynomial sensors, and other tracking models.  Early results for unbounded coefficients were obtained by Baras, Blankenship, and Hopkins \cite{BBH83}, in which the existence of a fundamental solution of the robust DMZ equation is derived for the one-dimensional case. A decisive step was taken by Yau and Yau \cite{YauYau08}. The existence and uniqueness of the weak solution of the DMZ equation were studied in classical Sobolev spaces through a standard variational approach, under mild assumptions which essentially require a greater growth rate of the observation function $h$ than the drift term $f$.  

In this paper, we will first study the well-posedness of the robust DMZ equation \eqref{eq:rdmz} under a weighted variational framework. The existence, uniqueness, strict positivity and robustness of the weak solution in some (exponentially) weighted Sobolev spaces will be derived. Based on the results for the robust DMZ equation and the inverse gauge (exponential) transformation, we prove the well-posedness of the stochastic DMZ equation \eqref{eq:dmz} and the Kushner-Stratonovich equation \eqref{eq:kse} in a \emph{buffered} weighted Sobolev spaces. Finally, the local Lipschitz robustness of the normalized conditional probability density with respect to the observation trajectory is obtained based on the strict positivity of the solution of the robust DMZ equation. This result shows the filter consistency in application scenarios, where instead of the whole continuous trajectory, the observations can only be collected at discrete time steps. 

The main advantage of the weighted variational approach proposed herein is that the well-posedness of the filtering equations can be established under more general and easily verifiable assumptions, and that the three important equations—namely, the robust DMZ equation, the stochastic DMZ equation, and the Kushner-Stratonovich equation—can all be studied within a unified class of buffered weighted Sobolev spaces. Intuitively, to ensure well-posedness of the filtering system \eqref{eq:model}, assumptions should be imposed either on the drift term $f(x)$ so that the state process is stable, or, more importantly, on the observation function $h$ so that sufficiently informative observations of the state process are available. Within the present weighted variational framework, these assumptions can be combined and coupled with each other in order to cover a broader class of systems in practical applications. Such a combination is essential, for instance, in many common filtering systems, including the well-known linear Gaussian case. Indeed, despite its significance in numerous applications, the case of detectable linear Gaussian systems with unstable state dynamics has remained largely incompatible with the well-posedness conditions imposed in most prior works on the DMZ equation (including its robust variants).

The organization of this paper is as follows.  Section~\ref{sec:framework} is devoted to constructing the weighted variational framework and define the weak solution of the robust DMZ equation under this framework. The well-posedness of the robust DMZ equation under this weighted variational framework is presented in Section \ref{sec:3}. The existence and uniqueness result of the stochastic DMZ equation and the Kushner-Stratonovich equation in buffered weighted spaces are studied in Section \ref{sec:4}. Useful sufficient conditions for well-posedness are summarized and illustrated through several classes of examples in Section \ref{sec:5}, and concluding remarks are given in Section \ref{sec:conclusion}.

%=====================================================================

\section{Weak solution of the robust DMZ equation under the weighted variational framework}\label{sec:framework}
%=====================================================================
In this section, we will first introduce the weighted Sobolev spaces, in which the robust DMZ equation is considered throughout this paper. Basic properties of the weighted Sobolev spaces will then be summarized. Finally, we will define the weak solution of the robust DMZ equation on these spaces based on the weighted variational framework. 

\subsection{Weighted Sobolev spaces}
\label{sec:3.1}
For given constants $\eta > 0$ and $p\geq 1$, let us define an exponential weight function on $\mathbb{R}^{n}$ as
\begin{equation}
	w_{\eta,p}(x) = \exp\left(2\eta (1 + |x|^{2})^{\frac{p+1}{2}}\right), \ x\in\mathbb{R}^{n}.
\end{equation}
Based on the weight function $w_{\eta,p}(x)$, the Hilbert space consisting of all square-integrable functions on $\mathbb{R}^{n}$ is denoted by
\begin{equation}\label{eq:wl2}
	H_{\eta,p}(\mathbb{R}^{n}):=L^2(w_{\eta,p}\,dx) = \left\{v\in L^{2}(\mathbb{R}^{n}): \int_{\R^n}|v|^2w_{\eta,p}\,dx < \infty\right\}
\end{equation}
with norm and inner products:
\begin{equation}
	\norm{v}_{\eta,p}^2:=\int_{\R^n}|v|^2w_{\eta,p}\,dx,\ (u,v)_{\eta,p} = \int_{\mathbb{R}^{n}}uvw_{\eta,p}dx,\ \forall\ u,v\in H_{\eta,p}(\mathbb{R}^{n}).
\end{equation}
With the weighted $L^{2}$ space $H_{\eta,p}(\mathbb{R}^{n})$, the weighted Sobolev space $V_{\eta,p}(\mathbb{R}^{n})$ is defined as:
\begin{equation}\label{eq:spaces}
	V_{\eta,p}(\mathbb{R}^{n}):=\big\{v\in H_{\eta,p}(\mathbb{R}^{n}):\ \nabla v\in (H_{\eta,p}(\mathbb{R}^{n}))^{n},\ (1+|x|^{2})^{\frac{p}{2}}v\in H_{\eta,p}(\mathbb{R}^{n})\big\},
\end{equation}
with norm
\[
\norm{v}_{V_{\eta,p}}^2:=\norm{\nabla v}_{\eta,p}^2+\norm{(1+|x|^{2})^{\frac{p}{2}}v}_{\eta,p}^2 .
\]
Henceforth, to simplify notation, we omit the explicit dependence on the state space $\mathbb{R}^{n}$ in the weighted spaces $H_{\eta,p}(\mathbb{R}^{n})$ and $V_{\eta,p}(\mathbb{R}^{n})$, and write them simply as $H_{\eta,p}$ and $V_{\eta,p}$. This abbreviation will not cause any ambiguity in the sequel.

As in the classical Sobolev spaces, the weighted spaces $V_{\eta,p}$, $H_{\eta,p}$, together with the dual space $V_{\eta,p}'$, form a Gelfand triple, as stated in the following lemma.
\begin{lemma}
	The triplet $(V_{\eta,p}, H_{\eta,p}, V_{\eta,p}')$ forms a Gelfand triple:
	\begin{equation}
		V_{\eta,p}\hookrightarrow H_{\eta,p}\cong H_{\eta,p}'\hookrightarrow V_{\eta,p}',
	\end{equation}
	that is, the embedding $V_{\eta,p} \hookrightarrow H_{\eta,p}$ is continuous and dense, and the duality pairing $\dual{\cdot}{\cdot}_{V_{\eta,p}',V_{\eta,p}}$ extends the inner product $\ip{\cdot}{\cdot}_{\eta,p}$ in $H_{\eta,p}$.
	\label{lem:2.1}
\end{lemma}

\begin{proof}
	Let $v\in V_{\eta,p}$. Notice that $(1+|x|^{2})^{p/2}\geq 1$ and thus,
	\begin{equation}
		\norm{v}_{\eta,p}\le\norm{(1+|x|^{2})^{\frac{p}{2}}v}_{\eta,p}^2\le\norm{v}_{V_{\eta,p}}
	\end{equation}
	Therefore, the embedding $V_{\eta,p} \hookrightarrow H_{\eta,p}$ is continuous. 
	
	In order to prove that the embedding is also dense, we only need to show that the space of smooth functions with compact support, $C_c^\infty(\mathbb{R}^{n})$, is dense in both $V_{\eta, p}$ and $H_{\eta,p}$. 
	
	Consider a function $\chi\in C_c^\infty(\mathbb{R}^{n})$, $\chi(x)\in[0,1]$, $\forall\ x\in\mathbb{R}^{n}$, which satisfies
	\begin{equation}
		\chi(x) = \left\{\begin{aligned}
			&1,\ |x|\leq 1, \\
			&0,\ |x|\geq 2.
		\end{aligned}\right.
	\end{equation}
	Let us define $\chi_R(x)=\chi(x/R)$, with a constant $R > 0$. Then we have
	\begin{equation}
		\nabla(\chi_{R}v) - \nabla v = (\chi_{R} - 1)\nabla v + v\nabla \chi_{R}.
	\end{equation}
	As $R\rightarrow\infty$, according to the dominated convergence theorem,
	\begin{equation}
		\lim\limits_{R\rightarrow\infty}\left\|(1+|x|^{2})^{p/2}(\chi_{R}v - v)\right\|_{\eta,p} = 0,  \ \lim\limits_{R\rightarrow\infty}\left\|(\chi_{R} - 1)\nabla v\right\|_{\eta,p} = 0, 
	\end{equation}
	and for some constant $C > 0$,
	\begin{equation}
		|v\nabla\chi_R|\le\frac{C}{R}|v|, \lim\limits_{R\rightarrow\infty}\|v\nabla\chi_{R}\|_{\eta,p}\leq \|v\|_{\eta,p}\lim\limits_{R\rightarrow\infty}\frac{C}{R} = 0. 
	\end{equation}
	Thus, the compactly supported elements in $V_{\eta,p}$ are dense. 
	
	Since the weight function $w_{\eta,p}(x)$ and $(1+|x|^{2})^{p/2}$ are both bounded on compact sets, for elements with compact support, the $V_{\eta,p}$ norm is equivalent to the norms in classical Sobolev spaces, and thus, the density of $C_c^\infty(\mathbb{R}^n)$ in the space of compactly supported elements of $V_{\eta,p}$ follows from standard mollification and $C_c^\infty$ approximation arguments.
	
	The same truncation-mollification procedure also yields the density of $C_c^\infty(\mathbb{R}^{n})$ in $H_{\eta,p}$. Therefore, the embedding $V_{\eta,p} \hookrightarrow H_{\eta,p}$ is also dense and we obtain the Gelfand triple:
	\begin{equation}
		V_{\eta,p}\hookrightarrow H_{\eta,p}\cong H_{\eta,p}'\hookrightarrow V_{\eta,p}'.
	\end{equation}
\end{proof}

%=====================================================================

\subsection{The weighted variational framework}
Based on the weighted spaces $H_{\eta,p}$ and $V_{\eta,p}$ defined in Section \ref{sec:3.1}, a weighted variational problem corresponding to the robust DMZ equation \eqref{eq:rdmz} can be formulated, and the weak solution of \eqref{eq:rdmz} can be defined by the weighted variational problem.

Given an observation path $Y = \{Y_{t}:0\leq t\leq T\}\in C([0,T];\R^m)$, let us define the time-varying bilinear form $a_{Y,\eta,p}:[0,T]\times V_{\eta,p}\times V_{\eta,p}\rightarrow\mathbb{R}$ by
\begin{equation}\label{eq:form}
	\begin{aligned} 
		a_{Y,\eta,p}(t;u,v):=& \frac12\int_{\mathbb{R}^{n}}\nabla u\cdot\nabla v\,w_{\eta,p}\,dx
		+\frac12\int_{\mathbb{R}^{n}} v\,\nabla u\cdot\nabla U_{\eta,p}\,w_{\eta,p}\,dx \\
		&-\int_{\mathbb{R}^{n}} (b_Y\cdot\nabla u)\,v\,w_{\eta,p}\,dx
		-\int_{\mathbb{R}^{n}} P_Y\,u\,v\,w_{\eta,p}\,dx,
	\end{aligned} 
\end{equation}
where $b_{Y}(t,x)$ and $P_{Y}(t,x)$ are defined in \eqref{eq:coeffs} and 
$$U_{\eta,p}(x) = 2\eta (1+|x|^{2})^{\frac{p+1}{2}}$$
is the exponential part of the weighted function $w_{\eta,p}$. 

Important regularity properties of the time-varying bilinear form $a_{Y,\eta,p}$ under mild assumptions are summarized in the following theorem.
\begin{theorem}
	\label{thm:2.2}
	Consider a continuous observation path $Y = \{Y_{t}:0\leq t\leq T\}\in C([0,T];\mathbb{R}^{m})$, and fixed constants $\eta > 0$ and $p\geq 1$. Assume that 
	\begin{enumerate}[label=\textup{(A\arabic*)}]
		\item\label{A2} There exists a constant $C_0>0$ such that for all $x\in\R^n$,
		\begin{gather*}
			|f(x)|\le C_0(1+|x|^{2})^{p/2},\quad |\nabla f(x)|\le C_0(1+|x|^{2})^{(p-1)/2},\\
			|h(x)|\le C_0(1+|x|^{2})^{p/2},\quad |\nabla h(x)|\le C_0(1+|x|^{2})^{(p-1)/2},\\ |\Delta h(x)|\leq C_{0}(1+|x|^{2})^{(2p-1)/2}.
		\end{gather*}
	\end{enumerate}
	then, the time-varying bilinear form $a_{Y,\eta,p}$ defined in \eqref{eq:form} satisfies the following properties:
	\begin{enumerate}
		\item \textbf{Boundedness and continuity}: There exists a constant $C=C(n,p,\eta,Y)$ such that
		\[
		|a_{Y,\eta,p}(t;u,v)|\;\le\;C\,\norm{u}_{V_{\eta,p}}\norm{v}_{V_{\eta,p}}\qquad\forall u,v\in V_{\eta,p},\ t\in[0,T],
		\]
		and $t\mapsto a_{Y,\eta,p}(t;u,v)$ is continuous for fixed $u,v\in V_{\eta,p}$.
		\item \textbf{Semi-coercivity}: If we further assume that:
		\begin{enumerate}[label=\textup{(A\arabic*)},start=2]
			\item\label{A3RDMZ} There exist constants $\beta_{\eta,p}>0$ and $C_{\eta,p}\ge0$ such that for all $x\in\mathbb{R}^{n}$,
			\begin{equation}\label{eq:A3RDMZ-R0}
				\begin{aligned} 
					-\frac12\divg f(x)&-\frac12|h(x)|^2
					+\frac12 f(x)\cdot\nabla U_{\eta,p}(x)\\
					&+\frac14|\nabla U_{\eta,p}(x)|^2+\frac14\Delta U_{\eta,p}(x) \\
					&\le C_{\eta,p}-\beta_{\eta,p}(1+|x|^{2})^{p}.
				\end{aligned} 
			\end{equation}
		\end{enumerate}
		then the bilinear form $a_{Y,\eta,p}$ is semi-coercive, in the sense that it satisfies the G\aa rding inequality:
		\begin{equation}\label{eq:garding}
			\begin{aligned} 
				a_{Y,\eta,p}(t;v,v)\ge&\frac12\norm{\nabla v}_{\eta,p}^2+\frac{\beta_{\eta,p}}{2}\norm{(1+|x|^{2})^{\frac{p}{2}}v}_{\eta,p}^2\\
				&-C\big(1+|Y_t|^{2p}\big)\norm{v}_{\eta,p}^2,\ \forall \ v\in V_{\eta,p}, \ t\in [0,T], 
			\end{aligned} 
		\end{equation}
		holds for some constant $C > 0$. 
	\end{enumerate}
\end{theorem}
\begin{proof}
	(i) According to the definition of the functions $b_{Y}(t,x)$ and $P_{Y}(t,x)$ in \eqref{eq:coeffs}, Assumption \ref{A2} yields the following pointwise bounds:
	\begin{equation}\label{eq:coeff-bounds}
		\begin{aligned} 
			&|\nabla U_{\eta,p}|\le2\eta(p+1)(1+|x|^{2})^{p/2},\qquad
			|b_Y(t,\cdot)|\le C_0(1+|Y_t|)(1+|x|^{2})^{p/2},\\
			&|P_Y(t,\cdot)|\le C_0\big(1+|Y_t|^2\big)(1+|x|^{2})^{p},
		\end{aligned} 
	\end{equation}
	which hold for all $x\in\mathbb{R}^{n}$. 
	
	Since $Y = \{Y_{t}:0\leq t\leq T\}$ is a continuous observation path and thus bounded in $[0,T]$, let us denote the supremum norm of the continuous path in $[0,T]$ by
	\begin{equation}
		M_{T}(Y) = \sup\limits_{t\in[0,T]}|Y_{t}| < \infty,
	\end{equation}
	then, by \eqref{eq:coeff-bounds} and the Cauchy-Schwartz inequality, we have
	\begin{align*}
		|a_{Y,\eta,p}(t;u,v)|&\le\tfrac12\norm{\nabla u}_{\eta,p}\norm{\nabla v}_{\eta,p}
		\\
		&+\big(\eta(p+1)+C_0(1+M_T(Y))\big)\norm{\nabla u}_{\eta,p}\norm{(1+|x|^{2})^{p/2}v}_{\eta,p}\\
		&\qquad+C_0(1+M_T(Y)^2)\norm{(1+|x|^{2})^{p/2}u}_{\eta,p}\norm{(1+|x|^{2})^{p/2}v}_{\eta,p},
	\end{align*}
	and each term on the right-hand side can be bounded by a finite multiple of the value $\|u\|_{V_{\eta,p}}\|v\|_{V_{\eta,p}}$.
	
	The continuity of $a_{Y,\eta,p}$ in $t$ follows from the continuity of $t\mapsto Y_t$ and the dominated convergence theorem. In fact, the continuity of $Y_t$ renders the integrand in the definition \eqref{eq:form} of $a_{Y,\eta,p}$ continuous with respect to $t$, and the domination functions can be obtained by estimating $b_Y(t,x)$ and $P_Y(t,x)$ by the bounds in \eqref{eq:coeff-bounds}.
	
	(ii) If $v\in C_{c}^{\infty}(\mathbb{R}^{n})$, according to the definition \eqref{eq:form} and the integration-by-part formula, we have
	\begin{equation}
		\begin{aligned} 
			&a_{Y,\eta,p}(t,v,v)  \\
			=&\frac{1}{2}\int_{\mathbb{R}^{n}}\nabla v\cdot \nabla v w_{\eta,p}dx + \frac{1}{2}\int_{\mathbb{R}^{n}}v\nabla v\nabla U_{\eta,p}w_{\eta,p} dx \\
			&-\int_{\mathbb{R}^{n}}(b_{Y}\cdot \nabla v) vw_{\eta,p}dx - \int_{\mathbb{R}^{n}} P_{Y}v^{2}w_{\eta,p} dx \\
			=&\frac{1}{2}\|\nabla v\|_{\eta,p}^{2} + \tfrac14\int_{\mathbb{R}^{n}}\nabla(v^2)\cdot\nabla\big(e^{U_{\eta,p}}\big)\,dx-\tfrac12\int_{\mathbb{R}^{n}} b_Y\cdot\nabla(v^2)\,w_{\eta,p} dx - \int_{\mathbb{R}^{n}} P_{Y}v^{2}w_{\eta,p} dx \\
			=& \frac{1}{2}\|\nabla v\|_{\eta,p}^{2}-\tfrac14\int_{\mathbb{R}^{n}} v^2\,\Delta\big(e^{U_{\eta,p}}\big)\,dx + \tfrac12\int_{\mathbb{R}^{n}} v^2\,\divg(b_Yw_{\eta,p})\,dx- \int_{\mathbb{R}^{n}} P_{Y}v^{2}w_{\eta,p} dx \\
			=& \frac{1}{2}\|\nabla v\|_{\eta,p}^{2}	-\int_{\mathbb{R}^{n}}\Big(\tfrac14|\nabla U_{\eta,p}|^2+\tfrac14\Delta U_{\eta,p}-\tfrac12\divg b_Y-\tfrac12 b_Y\cdot\nabla U_{\eta,p}+P_Y\Big)v^2w_{\eta,p}\,dx.
		\end{aligned} 
		\label{eq:24}
	\end{equation}
	Let us denote
	\begin{equation}
		R_{Y,\eta,p} := \tfrac14|\nabla U_{\eta,p}|^2+\tfrac14\Delta U_{\eta,p}-\tfrac12\divg b_Y-\tfrac12 b_Y\cdot\nabla U_{\eta,p}+P_Y.
	\end{equation}
	
	With the definition \eqref{eq:coeffs} of $b_{Y}(t,x)$ and $P_{Y}(t,x)$, we have
	\begin{equation}
		\begin{aligned}
			R_{Y,\eta,p}(t,x) =& -\frac12\divg f-\frac12|h|^2+\frac12f\cdot\nabla U_{\eta,p}+\frac14|\nabla U_{\eta,p}|^2+\frac14\Delta U_{\eta,p} \\
			&-f(x)\cdot(\nabla h(x))^\top Y_t +\frac12 |(\nabla h(x))^\top Y_t|^2 -\frac12 (\nabla h(x))^\top Y_t\cdot\nabla U_{\eta,p}(x).
		\end{aligned}
	\end{equation}
	Notice that
	\begin{equation}
		|\nabla U_{\eta,p}| = 2\eta (p+1)(1+|x|^{2})^{(p-1)/2} |x|\leq 2\eta (p+1)(1+|x|^{2})^{\frac{p}{2}}.
	\end{equation}
	Using the assumptions \ref{A2} and \ref{A3RDMZ}, we can obtain an upper bound of $R_{Y,\eta,p}$:
	\begin{equation}
		\begin{aligned}
			R_{Y,\eta,p}(t,x)\leq & C_{\eta,p} - \beta_{\eta,p}(1+|x|^{2})^{p} + C_{0}^{2}(1+|x|^{2})^{\frac{2p-1}{2}}|Y_{t}| + \frac{1}{2} C_{0}^{2}(1+|x|^{2})^{p-1} |Y_{t}|^{2} \\
			& + C_{0}\eta(p+1)(1+|x|^{2})^{\frac{2p-1}{2}}|Y_{t}| 
		\end{aligned}
	\end{equation}
	According to Young's inequality, (with the fact that for $p\geq 1$, the orders $k = \frac{2p}{2p-1}$ and $l = 2p$ satisfies $\frac{1}{k} + \frac{1}{l} = 1$),
	\begin{equation}
		\begin{aligned}
			&(C_{0}^{2}+ C_{0}\eta (p+1))(1+|x|^{2})^{\frac{2p-1}{2}}|Y_{t}| \leq  \frac{\beta_{\eta,p}}{4}(1+|x|^{2})^{p} + C  |Y_{t}|^{2p},
		\end{aligned}
	\end{equation}
	for some constant $C> 0$ depending on $\eta$ and $p$. 
	
	Also, with the fact that for $p > 1$, the orders $k = \frac{p}{p-1}$ and $l = p$ satisfies $\frac{1}{k} + \frac{1}{l} = 1$, according to Young's inequality, we have
	\begin{equation}
		\frac{1}{2} C_{0}^{2}(1+|x|^{2})^{p-1} |Y_{t}|^{2} \leq \frac{\beta_{\eta,p}}{4}(1+|x|^{2})^{p} + C |Y_{t}|^{2p}
		\label{eq:30}
	\end{equation}
	for some constant $C> 0$ depending on $\eta$ and $p$. Moreover, the inequality \eqref{eq:30} also holds for $p = 1$.
	
	Therefore, 
	\begin{equation}
		\begin{aligned} 
		R_{Y,\eta,p}(t,x) &\leq  C_{\eta,p} - \beta_{\eta,p}(1+|x|^{2})^{p} + \frac{\beta_{\eta,p}}{2}(1+|x|^{2})^{p} + C  |Y_{t}|^{2p} \\
		&\leq C(1+|Y_{t}|^{2p}) - \frac{\beta_{\eta,p}}{2}(1+|x|^{2})^{p}. 
		\end{aligned} 
		\label{eq:31}
	\end{equation}
	
	Taking the estimation \eqref{eq:31} back to \eqref{eq:24}, we obtain the G\aa rding inequality:
	\begin{equation}
		\begin{aligned} 
			a_{Y,\eta,p}(t,v,v)& \geq \frac{1}{2}\|\nabla v\|_{\eta,p}^{2} - \int_{\mathbb{R}^{n}}\left(C(1+|Y_{t}|^{2p}) - \frac{\beta_{\eta,p}}{2}(1+|x|^{2})^{p}\right) v^{2}w_{\eta,p}dx\\
			&=\frac12\norm{\nabla v}_{\eta,p}^2+\frac{\beta_{\eta,p}}{2}\norm{(1+|x|^{2})^{\frac{p}{2}}v}_{\eta,p}^2-C\big(1+|Y_t|^{2p}\big)\norm{v}_{\eta,p}^2,
		\end{aligned}
		\label{eq:32}
	\end{equation}
	which holds for all $v\in C_{c}^{\infty}(\mathbb{R}^{n})$. 
	
	For general $v\in V_{\eta,p}$, the G\aa rding inequality \eqref{eq:32} holds, because $C_{c}^{\infty}(\mathbb{R}^{n})$ is dense in $V_{\eta,p}$ according to Lemma \ref{lem:2.1} and the dominated convergence theorem according to the boundedness of $a_{Y,\eta,p}$ which is already proved in (i). 
\end{proof}

With the bilinear form $a_{Y,\eta,p}$, a weak (or variational) solution to the robust DMZ equation \eqref{eq:rdmz} in the weighted space $V_{\eta,p}$ can be defined as follows.
\begin{definition}[Weighted variational solution of the robust DMZ equation]
	\label{def:2.3}
	For given constants  $\eta > 0$ and $p\geq 1$, a weak  (or variational) solution of \eqref{eq:rdmz} on $[0,T]$ for the observation path $Y = \{Y_{t}:0\leq t\leq T\}$ is a function
	\[
	u\in L^2(0,T; V_{\eta,p})\cap C([0,T]; H_{\eta,p})\quad\text{with}\quad \partial_tu\in L^2(0,T; V_{\eta,p}')
	\]
	such that $u(0)=u_0$ and, for a.e.\ $t\in[0,T]$ and all $v\in V_{\eta,p}$,
	\begin{equation}\label{eq:varid}
		\dual{\partial_tu(t)}{v}_{V_{\eta,p}',V_{\eta,p}}+a_{Y,\eta,p}(t;u(t),v)=0,
	\end{equation}
	where $a_{Y,\eta,p}$ is the weighted time-varying bilinear form associated with \eqref{eq:rdmz}, defined in \eqref{eq:form} above.
\end{definition}

In the next section, we will study the well-posedness of the weighted variational solution of the robust DMZ equation, based on the properties of the bilinear form proved in Theorem \ref{thm:2.2}.

\section{Well-posedness of the weighted variational solution to the robust DMZ equation}\label{sec:3}

In this section, we establish the well-posedness of the weighted variational solution to the robust DMZ equation, which includes the existence, uniqueness, positivity, robustness with respect to the observation paths. The main result is stated in the following theorem.
\begin{theorem}[Well-posedness of the robust DMZ equation]\label{thm:rdmz}
	For fixed constants $\eta > 0$ and $p \geq 1$, assume that the initial value $u_0$ of the robust DMZ equation \eqref{eq:rdmz} belongs to the weighted space $H_{\eta,p}$. Then, under the regularity assumptions \textup{\ref{A2}} and  \textup{\ref{A3RDMZ}} as in Theorem \ref{thm:2.2}, the weighted variational solution of the robust DMZ equation \eqref{eq:rdmz}, which is introduced in Definition \ref{def:2.3}, is well-posed. That is:
	\begin{enumerate}[label=\textup{(\roman*)}]
		\item\label{rdmz-exist} For every continuous observation path $Y = \{Y_{t}:0\leq t\leq T\}\in C([0,T];\R^m)$ there exists a unique weighted variational solution $u=u^Y$ of \eqref{eq:rdmz}, and with $\beta_{\eta,p}$ from assumption \textup{\ref{A3RDMZ}}, we have the energy estimation:
		\begin{equation}\label{eq:energy-est}
			\begin{aligned} 
			\sup_{t\in[0,T]}\norm{u(t)}_{\eta,p}^2
			&+\int_0^T\Big(\norm{\nabla u(s)}_{\eta,p}^2+\beta_{\eta,p}\norm{(1+|x|^{2})^{p}u(s)}_{\eta,p}^2\Big)ds
			\\
			&\leq C\,\norm{u_0}_{\eta,p}^2 .
			\end{aligned} 
		\end{equation}
		where $C=C(n,m,\eta,p,Y,T)> 0$ is a generic constant. 
		\item\label{rdmz-pos} If the initial value $u_0\ge0$ a.e. and $\int_{\R^n} u_{0}(x) dx = 1$, then $u(t,\cdot)\ge0$ $a.e.$, and $\int_{\R^n} u(t,x)dx > 0$ for every $t\in[0,T]$.
		\item\label{rdmz-stab} The solution map is locally Lipschitz: for every $M>0$ there is $C_{M,T}>0$ such that for all observation paths $Y^{1}$ and $Y^{2}$ with supremum norm $$\sup\limits_{t\in[0,T]}\max\{|Y_{t}^{1}|, |Y_{t}^{2}|\}\le M,$$
		\begin{equation} 
			\begin{aligned} 
		\sup_{t\in[0,T]}\norm{u^{Y^1}(t)-u^{Y^2}(t)}_{\eta,p}^2&+\int_0^T\norm{u^{Y^1}(s)-u^{Y^2}(s)}_{V_{\eta,p}}^2\,ds \\
		&\leq C_{M,T}\,\norm{u_0}_{\eta,p}^2\,\sup_{t\in[0,T]}\big|Y^1_t-Y^2_t\big|^2.
		\end{aligned} 
		\end{equation} 
	\end{enumerate}
\end{theorem}
\begin{proof}
	(i) The existence and uniqueness of the weighted variational solution follow from the boundedness and semi-coercivity of the bilinear form $a_{Y,\eta,p}$, as proved in Theorem \ref{thm:2.2}. This is in fact a classical result due to Lions and Magenes \cite[Ch.~3, Sect.~1 and 4]{LionsMagenes72}. For the reader's convenience, we provide a modern formulation of this classical theorem for general Gelfand triples in the Appendix.
	
	For the energy estimation \eqref{eq:energy-est}, firstly, according to the G\aa rding inequality \eqref{eq:garding}, we have for a.e.\ $t\in[0,T]$,
	\begin{equation} 
		\begin{aligned}
			\frac12&\frac{d}{dt}\norm{u(t)}_{\eta,p}^2
			=\dual{\partial_tu(t)}{u(t)}_{V_{\eta,p}',V_{\eta,p}}=-a_{Y,\eta,p}(t;u(t),u(t))\\
			&\le-\frac12\norm{\nabla u(t)}_{\eta,p}^2-\frac{\beta_{\eta,p}}{2}\norm{(1+|x|^{2})^{\frac{p}{2}}u(t)}_{\eta,p}^2
			+C\big(1+|Y_t|^{2p}\big)\norm{u(t)}_{\eta,p}^2 .
		\end{aligned}
		\label{eq:34}
	\end{equation}
	Therefore, 
	\begin{equation}
		\frac{d}{dt}\norm{u(t)}_{\eta,p}^2\leq 2C\big(1+|Y_t|^{2p}\big)\norm{u(t)}_{\eta,p}^2 ,
	\end{equation}
	and according to Gronwall's inequality, 
	\begin{equation} 
		\begin{aligned} 
			\norm{u(t)}_{\eta,p}^2&\le \exp\left(2C \int_{0}^{t}(1+|Y_{s}|^{2p})ds\right)\norm{u_0}_{\eta,p}^2.
		\end{aligned} 
		\label{eq:37}
	\end{equation}
	In the meanwhile, with \eqref{eq:37} we may integrate the differential inequality \eqref{eq:34} over $[0,T]$, and obtain
	\begin{equation} 
		\begin{aligned}
			\int_0^T&\Big(\norm{\nabla u}_{\eta,p}^2+\beta_\eta\norm{(1+|x|^{2})^{\frac{p}{2}}u(t)}_{\eta,p}^2\Big)dt\\
			&\le\norm{u_0}_{\eta,p}^2+2C\int_0^T(1+|Y_t|^{2p})\norm{u(t)}_{\eta,p}^2\,dt\\
			&\le\norm{u_0}_{\eta,p}^2+\norm{u_0}_{\eta,p}^2\int_{0}^{T}2C(1+|Y_{t}|^{2p})\exp\left(2C \int_{0}^{t}(1+|Y_{s}|^{2p})ds\right)dt \\
			&= \norm{u_0}_{\eta,p}^2 + \norm{u_0}_{\eta,p}^2 \left(\exp\left(2C \int_{0}^{T}(1+|Y_{t}|^{2p})dt\right) - 1\right)
		\end{aligned}
		\label{eq:38}
	\end{equation} 
	Adding the estimations \eqref{eq:37} and \eqref{eq:38} gives the desired energy estimation \eqref{eq:energy-est}.
	
	(ii) The positivity of the weighted variational solution actually results from the maximum principle for the robust DMZ equation. We now proceed to prove this positivity. 
	
	Let $u^-(t,x)=\max(-u(t,x),0)$ denote the negative part of the solution. Since $u\in L^2(0,T;V_{\eta,p})$ and $\partial_{t}u\in L^{2}(0,T,V_{\eta,p}')$, also $u^-\in L^2(0,T;V_{\eta,p})$ and $\partial_{t}u^-\in L^2(0,T;V_{\eta,p}')$ with 
	\begin{equation} 
		\begin{aligned} 
			&u^- =-\mathbf 1_{\{u<0\}} u=\mathbf 1_{\{u<0\}} u^-,\  \partial_{t} u^-=-\mathbf 1_{\{u<0\}}\partial_{t} u=\mathbf 1_{\{u<0\}}\partial_{t} u^-,\\ &\nabla u^- = -\mathbf{1}_{\{u<0\}}\nabla u=\mathbf{1}_{\{u<0\}}\nabla u^-,
		\end{aligned} 
	\end{equation} 
	
	where $\mathbf 1_{\{u<0\}}$ is the indicator function. Therefore, since $u_{0}\geq 0$ a.e implies $u_{0}^{-} = 0$ a.e., we have
	\begin{equation} 
		\begin{aligned} 
			\frac12\norm{u^-(t)}_{\eta,p}^2&=\int_{0}^{t}\langle \partial_{s}u^-(s),u^-(s)\rangle_{V_{\eta,p}',V_{\eta,p}} ds = -\int_0^t\dual{\partial_su(s)}{u^-(s)}_{V_{\eta,p}',V_{\eta,p}} \,ds \\
			&=\int_0^ta_{Y,\eta,p}(s;u(s),u^-(s))\,ds,
		\end{aligned} 
	\end{equation}
	where in the last equality, we use the fact that $u$ is the weighted variational solution of the robust DMZ equation \eqref{eq:rdmz}. 
	
	Notice that 
	\begin{equation}
		\begin{aligned} 
			a_{Y,\eta,p}&(t;u,u^-)= \frac12\int_{\mathbb{R}^{n}}\nabla u\cdot\nabla u^-\,w_{\eta,p}\,dx
			+\frac12\int_{\mathbb{R}^{n}} u^-\,\nabla u\cdot\nabla U_{\eta,p}\,w_{\eta,p}\,dx \\
			&-\int_{\mathbb{R}^{n}} (b_Y\cdot\nabla u)\,u^-\,w_{\eta,p}\,dx
			-\int_{\mathbb{R}^{n}} P_Y\,u\,u^-\,w_{\eta,p}\,dx \\
			=&-\frac12\int_{\mathbb{R}^{n}}\mathbf 1_{\{u<0\}}\nabla u^-\cdot\nabla u^-\,w_{\eta,p}\,dx
			-\frac12\int_{\mathbb{R}^{n}} \mathbf 1_{\{u<0\}}u^-\,\nabla u^-\cdot\nabla U_{\eta,p}\,w_{\eta,p}\,dx \\
			&+\int_{\mathbb{R}^{n}} \mathbf 1_{\{u<0\}}(b_Y\cdot\nabla u^-)\,u^-\,w_{\eta,p}\,dx
			+\int_{\mathbb{R}^{n}} \mathbf 1_{\{u<0\}}P_Y\,u^-\,u^-\,w_{\eta,p}\,dx \\
			&=-a_{Y,\eta,p}(t,u^-,u^{-}).
		\end{aligned} 
	\end{equation}
	Thus, according to the G\aa rding inequality \eqref{eq:garding},
	\begin{equation} 
		\frac12\norm{u^-(t)}_{\eta,p}^2=-\int_0^ta_{Y,\eta,p}(s;u^-,u^-)\,ds
		\le C\int_0^t\big(1+|Y_s|^{2p}\big)\norm{u^-(s)}_{\eta,p}^2\,ds ,
	\end{equation} 
	and $u^-\equiv0$ follows from the Gronwall's inequality.
	
	The strict positivity of the integral:
	\begin{equation}
		\int_{\mathbb{R}^{n}}u(t,x) dx > 0,\ \forall\ t\in[0,T],
	\end{equation}
	stems from the classical result of strong maximum principle for parabolic equations in bounded domain. 
	
	In fact, since the initial value $u_{0}$ satisfies $\int_{\mathbb{R}^{n}}u_{0}(x) dx = 1$, we may find a point $x_{0}\in\mathbb{R}^{n}$ and a radius $r > 0$, such that $u_{0} > 0$ in the closed ball $B_{r}(x_{0}):\{x\in\mathbb{R}^{n}:|x-x_{0}|\leq r\}$, and we may consider the initial-boundary value problem of \eqref{eq:rdmz} on $B_{r}(x_{0})$, where all the coefficients of the parabolic equation \eqref{eq:rdmz} are bounded smooth functions due to the compactness of $B_{r}(x_{0})$. 
	
	The classical strong maximum principle for parabolic equations (cf. \cite{AronsonSerrin67}, for example) implies that $u(t,x) > 0$, for all $(t,x)\in[0,T]\times B_{r}(x_{0})$, and therefore,
	\begin{equation}
		\int_{\mathbb{R}^{n}}u(t,x)dx \geq \int_{B_{R_0}} u(t,x) dx > 0. 
	\end{equation}
	
	(iii) Let us denote $w=u^{Y^1}-u^{Y^2}$. Subtracting the variational identities \eqref{eq:varid}, we obtain 
	\begin{equation} 
		\dual{\partial_tw}{v}_{V_{\eta,p}',V_{\eta,p}}+a_{Y^1,\eta,p}(t,w,v)=\big(a_{Y^2,\eta,p}-a_{Y^1,\eta,p}\big)(t,u^{Y^2},v),\ \forall\  v\in V_{\eta,p} .
		\label{eq:43}
	\end{equation}
	Firstly, according to the G\aa rding inequality \eqref{eq:garding}, taking $v = w$ in \eqref{eq:43}, we have
	\begin{equation}
		\begin{aligned} 
			&\dual{\partial_tw}{w}_{V_{\eta,p}',V_{\eta,p}}+a_{Y^1,\eta,p}(t,w,w)\\
			&\quad \geq \frac{1}{2}\frac{d}{dt}\norm{w}_{\eta,p}^2+\frac{1}{2}\norm{\nabla w}_{\eta,p}^2  + \frac{\beta_{\eta,p}}{2} \norm{(1+|x|^{2})^{\frac{p}{2}}w}_{\eta,p}^{2}\\
			&\quad- C(1+|Y_{t}^1|^{2p})\|w\|_{\eta,p}^{2} \\
			&\quad\geq   \frac{1}{2}\frac{d}{dt}\norm{w}_{\eta,p}^2 + \gamma\norm{w}_{V_{\eta,p}}^2 - C(1+|Y_{t}^1|^{2p})\|w\|_{\eta,p}^{2},
		\end{aligned} 
		\label{eq:44}
	\end{equation}
	where $\gamma:=\min (\frac{1}{2},\frac{\beta_{\eta,p}}{2})$. 
	
	The two forms $a_{Y^1,\eta,p}$ and $a_{Y^2,\eta,p}$ differ only through the functions $b_Y(t,x)$ and $P_Y(t,x)$. According to the assumption \ref{A2}, for two observation paths $Y^1$ and $Y^2$ with $$\sup\limits_{t\in[0,T]}\max\{|Y_{t}^{1}|, |Y_{t}^{2}|\}\le M,$$
	we have the following estimations:
	\begin{equation} 
		\begin{aligned} 
			|b_{Y^1}(t,x)-b_{Y^2}(t,x)|&=|(\nabla h(x))^\top\delta Y_t|\le C_0(1+|x|^{2})^{\frac{p-1}{2}}|\delta Y_t|,\\
			|P_{Y^1}(t,x)-P_{Y^2}(t,x)|&\le |\delta Y_t| (|\Delta h(x)| + |\nabla h(x) f(x)|) + \frac{1}{2}|\nabla h(x)|^{2}|Y_{t}^{1} + Y_{t}^{2}| |\delta Y_{t}| \\
			&\leq \left((C_{0}^2 + C_{0})(1+|x|^{2})^{\frac{2p-1}{2}} + C_{0}^{2}(1+|x|^{2})^{p-1}M\right) |\delta Y_{t}|,
		\end{aligned} 
	\end{equation} 
	where $\delta Y_{t}:= Y_{t}^{1} - Y_{t}^{2}$. 
	
	Then, the right-hand side of \eqref{eq:43} can be bounded:
	\begin{equation}
		\begin{aligned}
			\biggl|\big(a_{Y^2,\eta,p}&-a_{Y^1,\eta,p}\big)(t,u^{Y^2},v)\biggr| \\
			&= \int_{\mathbb{R}^{n}}((b_{Y^1} - b_{Y^2})\cdot \nabla u^{Y^2}) vw_{\eta,p}dx + \int_{\mathbb{R}^{n}}(P_{Y^1} - P_{Y^2})u^{Y^2} vw_{\eta,p}dx \\
			\leq & C_{0}|\delta Y_{t}|\int_{\mathbb{R}^{n}}|\nabla u^{Y^2}|\cdot (1+|x|^{2})^{\frac{p-1}{2}} |v|w_{\eta,p} dx \\
			&+ \left(C_{0}^2 + C_{0}+ C_{0}^{2}M\right) |\delta Y_{t}|\int_{\mathbb{R}^{n}}(1+|x|^{2})^{\frac{p}{2}}|u^{Y^2}|\cdot (1+|x|^{2})^{\frac{p}{2}}|v|w_{\eta,p} dx \\
			\leq & C(1+M) |\delta Y_{t}|\norm{u^{Y^2}}_{V_{\eta,p}}\norm{v}_{V_{\eta,p}}
		\end{aligned}
	\end{equation}
	where $C > 0$ is a constant. In particular, let us take $v = w$ in \eqref{eq:43}, and we have
	\begin{equation}
		\begin{aligned} 
			&\dual{\partial_tw}{w}_{V_{\eta,p}',V_{\eta,p}}+a_{Y^1,\eta,p}(t,w,w)=\big(a_{Y^2,\eta,p}-a_{Y^1,\eta,p}\big)(t,u^{Y^2},w) \\
			&\quad\leq C(1+M) |\delta Y_{t}|\norm{u^{Y^2}}_{V_{\eta,p}}\norm{w}_{V_{\eta,p}} \leq \frac{\gamma}{2}\|w\|_{V_{\eta,p}}^{2} + C_{M,\gamma}|\delta Y_{t}|^{2}\|u^{Y^2}\|_{V_{\eta,p}}^{2},
		\end{aligned} 
		\label{eq:47}
	\end{equation}
	for some constant $C_{M,\gamma} > 0$, where we used Young's inequality.
	
	Combining \eqref{eq:44} and \eqref{eq:47}, we obtain 
	\begin{equation}
		\frac{1}{2}\frac{d}{dt}\norm{w}_{\eta,p}^2 + \frac{\gamma}{2}\norm{ w}_{V_{\eta,p}}^2 \leq  C(1+|Y_{t}^1|^{2p})\|w\|_{\eta,p}^{2} + C_{M,\gamma}|\delta Y_{t}|^{2}\|u^{Y^2}\|_{V_{\eta,p}}^{2}
		\label{eq:48}
	\end{equation}
	According to Gronwall's inequality and the energy estimation \eqref{eq:energy-est} of $u^{Y^2}$, we have
	\begin{equation} 
		\sup_{t\in[0,T]}\norm{u^{Y^1}(t)-u^{Y^2}(t)}_{\eta,p}^2\le\;C_{M,T}\,\norm{u_0}_{\eta,p}^2\,\sup_{t\in[0,T]}\big|Y^1_t-Y^2_t\big|^2,
		\label{eq:49}
	\end{equation}
	for some constant $C_{M,T} > 0$. The desired result is obtained by inserting \eqref{eq:49} into \eqref{eq:48} and integrating both sides over the time variable $t$.
\end{proof}

\section{The Buffered Spaces, DMZ Equation and Kushner-Stratonovich Equation}\label{sec:4}
In view of the well-posedness results for the robust DMZ equation in weighted Sobolev spaces $H_{\eta,p}$ and $V_{\eta,p}$, which are established in the preceding sections, we are now able to analyze the DMZ equation and the Kushner-Stratonovich equation, but in the \textit{buffered} weighted Sobolev spaces $H_{\eta',p}$ and $V_{\eta',p}$ for some $\eta' \in (0,\eta)$. Observe that the solution $u(t,x)$ of the robust DMZ equation \eqref{eq:rdmz} and the solution $\sigma(t,x)$ of the original stochastic DMZ equation \eqref{eq:dmz} are related by the gauge transformation \eqref{eq:gauge}. Here $\sigma(t,x)$ is the unnormalized conditional density of the conditional distribution $\pi_t$, which itself is the solution of the Kushner-Stratonovich equation \eqref{eq:kse}. 

Here, we call the weighted Sobolev spaces $H_{\eta',p}$ and $V_{\eta',p}$ \textit{buffered} spaces of $H_{\eta,p}$ and $V_{\eta,p}$, respectively, if $0 < \eta' < \eta$. This terminology is motivated by the monotonicity of the weight function with respect to the parameter $\eta$. It is then straightforward that
\begin{equation}
	H_{\eta,p} \subset H_{\eta',p}, \qquad V_{\eta,p} \subset V_{\eta',p}, \qquad \forall \, 0 < \eta' < \eta,
\end{equation}
and the embeddings $H_{\eta,p} \hookrightarrow H_{\eta',p}$ and $V_{\eta,p} \hookrightarrow V_{\eta',p}$ are continuous.

\subsection{Existence of Weak Solutions to the DMZ Equation and Kushner-Stratonovich Equation in Buffered Spaces}
\label{sec:4.1}
Firstly, let us study the existence of a weak solution to the DMZ equation \eqref{eq:dmz} based on the well-posedness result of its robust version. 

Under the assumptions \ref{A2} and \ref{A3RDMZ}, the robust DMZ equation \eqref{eq:rdmz} admits a unique weighted variational solution $u\in L^{2}(0,T;V_{\eta,p})\cap C([0,T];H_{\eta,p})$ with $\partial_{t}u\in L^{2}(0,T;V_{\eta,p}')$. The existence of a solution to the original DMZ equation \eqref{eq:dmz} is now stated and proved in a constructive way.
\begin{theorem}\label{thm:zakai}
	For given parameters $\eta > 0$ and $p\geq 1$, assume that the conditions \ref{A2} and \ref{A3RDMZ} holds, such that for each continuous observation path $Y = \{Y_{t}:0\leq t\leq T\}\in C([0,T],\mathbb{R}^{m})$, a weighted variational solution $u^{Y}\in L^{2}(0,T;V_{\eta,p})\cap C([0,T];H_{\eta,p})$ with $\partial_{t}u^{Y}\in L^{2}(0,T;V_{\eta,p}')$ exists for the robust DMZ equation \eqref{eq:rdmz}. Then, the function $\sigma(t,x)$ obtained by the inverse exponential transformation:
	\begin{equation}
		\sigma(t,x) := e^{Y_{t}^{\top} h(x)}u^{Y}(t,x),\ (t,x)\in [0,T]\times\mathbb{R}^{n},
		\label{eq:51}
	\end{equation}
	is a \textit{buffered} weak solution to the original stochastic DMZ equation \eqref{eq:dmz}, in the sense that, almost surely,
	\begin{enumerate}[label=\textup{(\roman*)}]
		\item\label{zakai-reg} For every $\eta'\in(0,\eta)$, $\sigma\in C([0,T];H_{\eta',p})$.
		\item\label{zakai-weak} For every $t \in [0,T]$ and every $\varphi \in C^2(\mathbb{R}^n)$ with $\varphi$ and its partial derivatives up to second-order growing at most polynomially as $|x| \to \infty$,
		\begin{equation}\label{eq:zakai-weak}
			\sigma_t(\varphi)=\sigma_0(\varphi)+\int_0^t\sigma_s(L\varphi)\,ds+\int_0^t\sigma_s(\varphi h^{\top})dY_s ,
		\end{equation}
		where the last integral is an It\^o integral with respect to the semi-martingale $Y$ and $$\sigma_{t}(\varphi) = \int_{\R^n}\varphi(x)\sigma(t,x)dx.$$
	\end{enumerate}
\end{theorem}

\begin{proof}
	According to the inverse exponential transformation \eqref{eq:51} and the growth rate condition of $h(x)$ in the assumption \ref{A2}, for a given continuous observation path $Y$, 
	\begin{equation}
		\begin{aligned} 
			\sigma^2(t,x)w_{\eta',p}(x)  &= (u^{Y}(t,x) )^{2}\exp\left(2Y_{t}^{\top}h(x) + 2\eta'(1+|x|^{2})^{\frac{p+1}{2}}\right) \\
			&\leq (u^{Y}(t,x))^{2} \exp\left(2C_{0}M_{T}(Y)(1+|x|^{2})^{\frac{p}{2}} + 2\eta' (1+|x|^{2})^{\frac{p+1}{2}}\right),
		\end{aligned} 
	\end{equation}
	where $M_{T}(Y) = \sup\limits_{t\in[0,T]}|Y_{t}| < \infty$. Thus,
	\begin{equation}
		\sigma^{2}(t,x)w_{\eta',p}(x) \leq (u^{Y}(t,x))^{2}\exp\left(2\eta (1+|x|^{2})^{\frac{p+1}{2}}\right) = (u^{Y}(t,x))^{2}w_{\eta,p}(x),
		\label{eq:54b}
	\end{equation}
	holds for $|x|$ sufficiently large, such that
	\begin{equation}
		(1+|x|^{2})^{\frac{1}{2}}\geq \frac{C_{0}M_{T}(Y)}{\eta - \eta'}. 
		\label{eq:55}
	\end{equation}
	Since $u^{Y}\in C([0,T],H_{\eta,p})$, the right-hand side of \eqref{eq:54b} is integrable, and therefore,
	\begin{equation}
		\lim\limits_{R\rightarrow\infty}\int_{|x|\geq R}\sigma^{2}(t,x)w_{\eta',p}(x) dx \leq \lim\limits_{R\rightarrow\infty}\int_{|x|\geq R}(u^{Y}(t,x))^{2}w_{\eta,p} dx  = 0,
	\end{equation}
	which implies
	\begin{equation}
		\|\sigma(t,\cdot)\|_{\eta',p}^{2} = \int_{\mathbb{R}^{n}}\sigma^{2}(t,x)w_{\eta',p}(x) dx < \infty,\ \forall \ t\in [0,T]. 
	\end{equation}
	Moreover, for every $0\leq  s < t\leq T$,
	\begin{equation}
		\begin{aligned}
			\|\sigma(t,\cdot) - \sigma(s,\cdot)\|_{\eta',p}^{2} =&  \int_{\mathbb{R}^{n}} (\sigma(t,x) - \sigma(s,x))^{2} w_{\eta',p}(x) dx \\
			\leq & 2 \int_{\mathbb{R}^{n}}(u^{Y}(t,x))^{2}\left(\exp(2Y_{t}^{\top}h(x)) - \exp(2Y_{s}^{\top}h(x))\right)w_{\eta',p}(x) dx\\
			&+ 2\int_{\mathbb{R}^{n}}\exp(2Y_{s}^{\top}h(x))(u^{Y}(t,x) - u^{Y}(s,x))^{2} w_{\eta',p}(x) dx.
		\end{aligned}
		\label{eq:58}
	\end{equation}
	For the first integral on the right-hand side of \eqref{eq:58}, since the observation path $Y$ is continuous,
	\begin{equation}
		\lim\limits_{t\rightarrow s}(u^{Y}(t,x))^{2}\left(\exp(2Y_{t}^{\top}h(x)) - \exp(2Y_{s}^{\top}h(x))\right)w_{\eta',p}(x) = 0,\ a.e.\ x\in\mathbb{R}^{n},
	\end{equation}
	and 
	\begin{equation}
		\begin{aligned}
			&|(u^{Y}(t,x))^{2}\exp(2(Y_{t}-Y_{s})^{\top}h(x))w_{\eta',p}(x) |\\
			&\leq (u^{Y}(t,x))^{2}\exp\left(4C_{0}M_{T}(Y)(1+|x|^{2})^{\frac{p}{2}} + 2\eta'(1+|x|^{2})^{\frac{p+1}{2}}\right) \\
			&\leq (u^{Y}(t,x))^{2} w_{\eta,p}(x),\ \forall \ x\in\mathbb{R}^{n}\ s.t., \  (1+|x|^{2})^{\frac{1}{2}}\geq \frac{2C_{0}M_{T}(Y)}{\eta - \eta'}.
		\end{aligned}
	\end{equation}
	According to the dominated convergence theorem, we have
	\begin{equation}
		\lim\limits_{t\rightarrow s}\int_{\mathbb{R}^{n}}(u^{Y}(t,x))^{2}\left(\exp(2Y_{t}^{\top}h(x)) - \exp(2Y_{s}^{\top}h(x))\right)w_{\eta',p}(x) dx = 0. 
		\label{eq:61}
	\end{equation}
	For the second integral on the right-hand side of \eqref{eq:58}, 
	\begin{equation}
		\int_{\mathbb{R}^{n}}\exp(2Y_{s}^{\top}h(x))(u^{Y}(t,x) - u^{Y}(s,x))^{2} w_{\eta',p}(x) dx\leq C\|u^{Y}(t,\cdot) - u^{Y}(s,\cdot)\|_{\eta,p}^{2},
	\end{equation}
	where $C > 0$ is a constant and can be chosen as
	\begin{equation}
		C = \exp\left(2C_{0}M_{T}(Y)\left(\frac{C_{0}M_{T}(Y)}{\eta - \eta'}\right)^{p}\right),
	\end{equation}
	since the estimate \eqref{eq:54b} is valid for sufficiently large $|x|$ in the sense that the inequality \eqref{eq:55} is satisfied.
	
	Because of the continuity result $u^{Y}\in C([0,T];H_{\eta,p})$, we have
	\begin{equation}
		\begin{aligned} 
		\lim\limits_{t\rightarrow s}&\int_{\mathbb{R}^{n}}\exp(2Y_{s}^{\top}h(x))(u^{Y}(t,x) - u^{Y}(s,x))^{2} w_{\eta',p}(x) dx \\
		&\leq C\lim\limits_{t\rightarrow s}\|u^{Y}(t,\cdot) - u^{Y}(s,\cdot)\|_{\eta,p}^{2} = 0.
		\end{aligned} 
		\label{eq:64}
	\end{equation}
	
	Taking \eqref{eq:61} and \eqref{eq:64} back to \eqref{eq:58}, we prove the desired continuity result
	\begin{equation}
		\lim\limits_{t\rightarrow s}\|\sigma(t,\cdot) - \sigma(s,\cdot)\|_{\eta',p}^{2} = 0,\ \forall\ 0\leq s < t\leq T,
	\end{equation}
	and thus $\sigma\in C([0,T];H_{\eta',p})$, for every $\eta'\in (0,\eta)$. 
	
	For part \ref{zakai-weak}, let us first restrict the test function $\varphi$ to smooth functions with compact support in \eqref{eq:zakai-weak}. For a fixed $\varphi\in C_{c}^{\infty}(\mathbb{R}^{n})$, let us define 
	\[
	\psi_t(x) :=e^{Y_t^{\top}h(x)}\varphi(x)\,w_{\eta,p}^{-1}(x),
	\]
	which is also a smooth function with compact support for each $t\in[0,T]$. According to It\^o's formula, for each $x\in\mathbb{R}^{n}$, we have $a.s.$,
	
	\begin{equation}\label{eq:psi-ito}
		\psi_t(x)=\psi_0(x)+\int_0^t\psi_s(x)\,h^{\top}(x) dY_s+\frac12\int_0^t\psi_s(x)\,|h(x)|^2\,ds .
	\end{equation}
	Notice that for each $t\in[0,T]$, the functions $\psi_{t}(x)$, $\psi_{t}(x)h(x)$ and $\psi_{t}(x) |h(x)|^{2}$ are all contained in $C_{c}^{\infty}(\mathbb{R}^{n})$, and therefore, are elements of $V_{\eta,p}$. We claim that the following product rule:
	\begin{equation}\label{eq:product}
		\begin{aligned} 
			\ip{u_t}{\psi_t}_{\eta,p}=\ip{u_0}{\psi_0}_{\eta,p}
			&+\int_0^t\dual{\partial_su_s}{\psi_s}_{V_{\eta,p}',V_{\eta,p}}\,ds
			+\int_0^t\ip{u_s}{\psi_sh^{\top}}_{\eta,p} dY_s \\
			&+\frac12\int_0^t\ip{u_s}{\psi_s|h|^2}_{\eta,p}\,ds .
		\end{aligned} 
	\end{equation}
	holds for all $t\in [0,T]$. 
	In fact, let $0=t_0<\dots<t_N=t$ be a partition of mesh $\delta = t_{k+1} - t_{k}$, $k=0,\cdots,N-1$. Thus,
	\[
	\ip{u_t}{\psi_t}_{\eta,p}-\ip{u_0}{\psi_0}_{\eta,p}
	=\sum_{k=0}^{N-1}\ip{u_{t_{k+1}}-u_{t_k}}{\psi_{t_{k+1}}}_{\eta,p}
	+\sum_{k=0}^{N-1}\ip{u_{t_k}}{\psi_{t_{k+1}}-\psi_{t_k}}_{\eta,p} .
	\]
	Since $\partial_t u\in L^2(0,T;V_{\eta,p}')$ , and the function $s\rightarrow\psi_{s}$ is continuous in $V_{\eta,p}$, the limit of the first sum equals
	\[
	\lim\limits_{\delta\rightarrow 0} \sum_{k=0}^{N-1}\ip{u_{t_{k+1}}-u_{t_k}}{\psi_{t_{k+1}}}_{\eta,p} = \int_0^t\dual{\partial_su_s}{\psi_s}_{V_{\eta,p}',V_{\eta,p}}\,ds,
	\]
	For the second sum, according to the definition of It\^o's integral, we may insert \eqref{eq:psi-ito}:
	\[
	\lim\limits_{\delta\rightarrow 0}\sum_{k=0}^{N-1}\ip{u_{t_k}}{\psi_{t_{k+1}}-\psi_{t_k}}_{\eta,p}
	=\int_0^t\ip{u_{s}}{\psi_sh^{\top}}_{\eta,p} dY_s
	+\frac12\int_0^t\ip{u_{s}}{\psi_s|h|^2}_{\eta,p}\,ds ,
	\]
	
	Next, since $u$ is the weighted variational solution of the robust DMZ equation \eqref{eq:rdmz}, for all $s\in[0,T]$,
	\begin{equation} 
		\begin{aligned} 
			\dual{\partial_su_s}{\psi_s}_{V_{\eta,p}',V_{\eta,p}}= &-a_{Y,\eta,p}(s;u_s,\psi_{s})
			\label{eq:68a}
		\end{aligned} 
	\end{equation} 
	Notice that $\psi_{s}(x)  = e^{Y_{t}^{\top}h(x)}\varphi(x) w_{\eta,p}^{-1}(x)$,
	\begin{equation}
		\begin{aligned}
			\nabla (\psi_{s}(x) w_{\eta,p}(x)) =&\varphi(x) e^{Y_{t}^{\top}h(x)}Y_{t}^{\top}\nabla h(x) + e^{Y_{t}^{\top}h(x)}\nabla\varphi(x), \\
			\Delta (\psi_{s}(x) w_{\eta,p}(x)) = &e^{Y_{t}^{\top}h(x)}Y_{t}^{\top}\nabla h(x) \nabla\varphi(x) + \varphi(x) e^{Y_{t}^{\top}h(x)}\left(Y_{t}^{\top}\Delta h(x) + |\nabla h(x) Y_{t}|^{2}\right) \\
			&+ e^{Y_{t}^{\top}h(x)}\Delta \varphi(x) + e^{Y_{t}^{\top}h(x)}Y_{t}^{\top}\nabla h(x) \nabla \varphi(x). 
		\end{aligned}
		\label{eq:69}
	\end{equation}
	Based on the integration-by-part formula (since $\psi_{s}\in C_{c}^{\infty}$), take \eqref{eq:69} back into the definition of $a_{Y,\eta,p}$, and we have
	\begin{equation}
		\dual{\partial_su_s}{\psi_s}_{V_{\eta,p}',V_{\eta,p}} = \int_{\mathbb{R}^{n}}u(t,x)e^{Y_{t}^{\top}h(x)}\left(\frac{1}{2}\Delta \varphi(x) + f(x)\cdot\nabla \varphi(x) -\frac{1}{2}|h(x)|^{2}\varphi(x)\right)
		\label{eq:70}
	\end{equation}
	Observe that
	\begin{equation} 
		\begin{aligned} 
			\ip{u_t}{\psi_t}_{\eta,p}=\int_{\mathbb{R}^{n}} u(t,x)\,e^{Y_t^{\top}h(x)}\varphi(x)\,dx=\sigma_t(\varphi),
			\\
			\ip{u_t}{\psi_th}_{\eta,p}=\sigma_t(\varphi h),\ 
			\ip{u_t}{\psi_t|h|^2}_{\eta,p}=\sigma_t(\varphi|h|^2).
		\end{aligned} 
		\label{eq:68}
	\end{equation}
	Substituting \eqref{eq:70} and \eqref{eq:68}into \eqref{eq:product}, we obtain the desired \eqref{eq:zakai-weak} for all $\varphi\in C_{c}^{\infty}(\mathbb{R}^{n})$. 
	
	The identity \eqref{eq:zakai-weak} holds for all $\varphi\in C^{2}(\mathbb{R}^{n})$ with $\varphi$ and its partial derivatives up to second-order growing at most polynomially as $|x|\rightarrow\infty$, because of the dominated convergence theorem. In fact, for each $\varphi\in C^{2}(\mathbb{R}^{n})$, we may find a sequence $\{\varphi_{n}(x)\}_{n=1}^{\infty}\subset C_{c}^{\infty}(\mathbb{R}^{n})$, such that
	\begin{equation}
		\lim\limits_{n\rightarrow\infty}\varphi_{n}(x) = \varphi(x),\ \lim\limits_{n\rightarrow\infty}\nabla\varphi_{n}(x) = \nabla\varphi(x), \ \lim\limits_{n\rightarrow\infty}\nabla^{2}\varphi_{n}(x) = \nabla^{2}\varphi(x),
	\end{equation}
	pointwisely for all $x\in\mathbb{R}^{n}$, based on the standard mollification method, and there exists a constant $C > 0$ and $q\in\mathbb{N}$, such that 
	\begin{equation}
		\max\left\{	|\varphi_{n}(x)|,\ |\nabla\varphi_{n}(x)|,\ |\nabla^{2}\varphi_{n}(x)|\right\}\leq C (1+|x|^{q}),\ \forall\ n\in\mathbb{N},\ x\in\mathbb{R}^{n}. 
	\end{equation}
	
	Since $\sigma_{t}\in C([0,T],H_{\eta',p})$, the polynomial growth condition implies that
	\begin{equation}
		\begin{aligned} 
		\int_{\mathbb{R}^{n}}\sigma(t,x)\varphi(x)dx &\leq C \int_{\mathbb{R}^{n}}\sigma(t,x)|x|^{r} dx \\
		&\leq \tilde{C}\|\sigma(t,\cdot)\|_{\eta',p}\left(\int_{\R^n}|x|^{2r}w_{\eta',p}^{-1}(x)dx\right)^{\frac{1}{2}} < \infty,
		\end{aligned} 
		\label{eq:74}
	\end{equation}
	where $r > 0$, $C > 0$ and $\tilde{C} > 0$ are some constants, and we used the Cauchy-Schwartz inequality. The same results hold for \eqref{eq:74}, if we substitute $\varphi(x)$ by its first and second order derivatives. 
\end{proof}

Formally, the solution $\sigma(t,x)$ of the DMZ equation \eqref{eq:dmz} is an unnormalized version of the conditional probability density function $\pi(t,x)$, that is,
\begin{equation}
	\pi(t,x) = \frac{\sigma(t,x)}{\sigma_{t}(\mathbf{1})},\ \text{with} \ \sigma_{t}(\mathbf{1}) = \int_{\R^n}\sigma(t,x) dx. 
	\label{eq:75}
\end{equation}
In order to make \eqref{eq:75} meaningful, it is required that the integral $\sigma_{t}(\mathbf{1}) > 0$, for all $t\in [0,T]$. Notice that the strict positivity of $\int_{\mathbb{R}^{n}}u^{Y}(t,x)dx > 0$ is proved in Theorem \ref{thm:rdmz}. Therefore, the boundedness of a continuous observation path, $Y_{t}$, implies that at least in a small closed ball, the unnormalized density function $\sigma(t,x)$ is strictly positive, and 
\begin{equation}
	\sigma_{t}(\mathbf{1}) = \int_{\R^n}u^{Y}(t,x) e^{Y_{t}^{\top} h(x)} dx > 0
\end{equation}
holds for all $t\in[0,T]$. 

Applying It\^o's formula to \eqref{eq:75}, we obtain the Kushner-Stratonovich equation satisfied by the normalized conditional density function $\pi(t,x)$ in the weak form:
\begin{equation}
	\pi_t(\varphi)=\pi_0(\varphi)+\int_0^t\pi_s(L\varphi)ds+
	\int_0^t[\pi_s(\varphi h)-\pi_s(\varphi)\pi_s(h)]\cdot[dY_s-\pi_s(h)ds],
\end{equation} 
for all $\varphi\in C^{2}(\mathbb{R}^{n})$ with $\varphi$ and its partial derivatives up to second-order growing at most polynomially as $|x|\rightarrow\infty$. 

Up to now, we have derived the existence result of a weak solution to the DMZ equation and Kushner-Stratonovich equation in buffered spaces, from the well - posedness of the robust DMZ equation \eqref{eq:rdmz}. 

\subsection{Uniqueness Result of the DMZ Equation and K-S Equation}
Generally speaking, the uniqueness result of the DMZ equation and Kushner-Stratonovich equation stems from the uniqueness result of the robust DMZ equation in a buffered space. Therefore, we need to restate the condition \ref{A3RDMZ} in Theorem \ref{thm:rdmz}, in order to make it compatible with this buffered setting:
\begin{enumerate}[label=\textup{(A\arabic*)},start=2]
	\item[(A2-B)]\label{A3RDMZ-b} Fix a constant $\eta* > 0$ and $p\geq 1$, for all $\eta\in (0,\eta^{*})$, there exist constants $\beta_{\eta,p}>0$ and $C_{\eta,p}\ge0$ such that for all $x\in\mathbb{R}^{n}$,
	\begin{equation}\label{eq:A3RDMZ-1}
		\begin{aligned} 
			-\frac12\divg f(x)&-\frac12|h(x)|^2
			+\frac12 f(x)\cdot\nabla U_{\eta,p}(x)+\frac14|\nabla U_{\eta,p}(x)|^2+\frac14\Delta U_{\eta,p}(x) \\
			&\le C_{\eta,p}-\beta_{\eta,p}(1+|x|^{2})^{p}.
		\end{aligned} 
	\end{equation}
\end{enumerate}

The weak solution of Kushner-Stratonovich equation in buffered spaces is defined to be an element in the following function class:
\begin{definition}[Buffered weighted Sobolev Kushner-Stratonovich density class]\label{def:ks-sobolev-class}
	For fixed constants $\eta > 0$ and $p\geq 1$, the class $\mathfrak A_{\eta,p}(Y,T)$ consists of observation-filtration progressively measurable density-valued processes $\pi_t(dx)=\pi(t,x)dx$ in the stochastic filtering setting such that, a.s.:
	\begin{enumerate}[label=\textup{(\roman*)}]
		\item $\pi(t,x)\ge0$ and $\int_{\mathbb{R}^{n}} \pi(t,x)dx=1$ for every $t\in[0,T]$;
		\item $\pi_{t}\in C([0,T]; H_{\eta,p})\cap L^{2}(0,T;V_{\eta,p})$;
		\item for every $\varphi\in C^{2}(\R^n)$ with $\varphi$ and its partial derivatives up to second-order growing at most polynomially as $|x|\rightarrow\infty$, the Kushner-Stratonovich equation
		\begin{equation}
			\pi_t(\varphi)=\pi_0(\varphi)+\int_0^t\pi_s(L\varphi)ds+
			\int_0^t[\pi_s(\varphi h^{\top})-\pi_s(\varphi)\pi_s(h^{\top})][dY_s-\pi_s(h)ds],
		\end{equation} 
		holds, where $L$ is the infinitesimal generator of the state process in the filtering system \eqref{eq:model}. 
	\end{enumerate}
\end{definition}
The existence result in Section \ref{sec:4.1} implies that the function class $\mathfrak{A}_{\eta,p}(Y,T)$ is not empty for $\eta>0$ and $p\geq 1$ compatible with the assumption (A2-B). The uniqueness result of Kushner-Stratonovich equation is stated as the following theorem:
\begin{theorem}\label{thm:kse}
	For fixed constants $\eta^{*} > 0$ and $p\geq 1$, assume that the conditions (A1) and (A2-B) hold, such that there exists a weighted variational solution to the robust DMZ equation \eqref{eq:rdmz} and thus the solution class $\mathfrak{A}_{\eta,p}(Y,T)$ of the Kushner-Stratonovich equation \eqref{eq:kse} is not empty for $\eta\in(0,\eta^{*})$. 
	
	If two elements $\pi^{1}(t,x)$ and $\pi^{2}(t,x)$ of $\mathfrak A_{\eta,p}(Y,T)$ have the same initial density $$\pi^{1}(0,x) = \pi^{2}(0,x),\ a.e.\ x\in\mathbb{R}^{n},$$ 
	then they coincide. Thus, uniqueness holds in the explicitly defined buffered weighted Sobolev Kushner-Stratonovich density class $\mathfrak{A}_{\eta,p}(Y,T)$.
\end{theorem}
\begin{proof}
	According to Theorem \ref{thm:rdmz}, under the assumptions (A1) and (A2-B), for an arbitrary  $\eta_{1}\in(0,\eta^{*})$, and a continuous observation path $Y$, there exists a weighted variational solution $u^{Y}\in L^{2}(0,T;V_{\eta_{1},p})\cap C([0,T];H_{\eta_{1},p})$. 
	
	Next, for arbitrary $\eta_{2}\in(0,\eta_{1})$, there exists $\sigma\in C([0,T];H_{\eta_{2},p})$ which solves the stochastic DMZ equation \eqref{eq:dmz}, and also its normalized version $\pi\in C([0,T];H_{\eta_{2},p})$, which is an element in the buffered weighted Sobolev Kushner-Stratonovich density class $\mathfrak{A}_{\eta_{2},p}(Y,T)$. 
	
	Assume that $\pi^1(t,x),\pi^2(t,x)\in\mathfrak A_{\eta_2,p}(Y,T)$ have the same initial density, $$\pi^1(0,x) = \pi^2(0,x), a.e. x\in\mathbb{R}^{n}.$$ Then, corresponding to each element $\pi^{i}$, we can reconstruct solutions $\sigma^{i}$ of the DMZ equation, and $u^{i}$ of the robust DMZ equation, respectively:
	\begin{equation} 
		\begin{aligned} 
		&\sigma^i(t,x)= \exp\left(\int_{0}^{t}\pi^{i}_{s}(h^{\top})dY_{s} - \frac{1}{2}\int_{0}^{t}|\pi^{i}_{s}(h)|^{2}ds\right)\pi^{i}(t,x),
		\\
		& u^i(t,x)=e^{-Y_{t}^{\top}h(x)}\sigma^i(t,x),\ i=1,2 .
		\end{aligned} 
		\label{eq:82}
	\end{equation} 
	Notice that for a fixed continuous observation path $Y$, the exponential martingale 
	\begin{equation}
		Z_{t}^{i} := \exp\left(\int_{0}^{t}\pi^{i}_{s}(h^{\top})dY_{s} - \frac{1}{2}\int_{0}^{t}|\pi^{i}_{s}(h)|^{2}ds\right), \ i=1,2,
	\end{equation}
	is bounded. According to the definition of $\mathfrak{A}_{\eta_{2},p}(Y,T)$, $$\pi^{i}\in C([0,T];H_{\eta_{2},p})\cap L^{2}(0,T;V_{\eta_{2},p}),$$ and thus,
	\begin{equation}
		\sigma^{i}\in C([0,T];H_{\eta_{2},p})\cap L^{2}(0,T;V_{\eta_{2},p}),\ i=1,2. 
	\end{equation}
	
	With the same estimation method in the proof of Theorem \ref{thm:zakai}, we can show that for each $\eta_{3}\in(0,\eta_{2})$, the reconstructed solution $u^{i}$ of the robust DMZ equation satisfies:
	\begin{equation}
		u^{i}\in C([0,T];H_{\eta_{3},p})\cap L^{2}(0,T;V_{\eta_{3},p}),\ i=1,2. 
	\end{equation}
	The property that $u^{i}(t,x)$ satisfies the robust DMZ equation \eqref{eq:varid} stems from a direct computation with It\^o's formula, since all the required regularity conditions are satisfied. Therefore, $u^{i}(t,x)$, $i=1,2$, are both the weighted variational solutions to the robust DMZ equation with buffered coefficient $\eta_{3}\in (0,\eta)$. 
	
	Based on the uniqueness result of robust DMZ equation and the assumption (A2-B), the uniqueness also holds in buffered weighted Sobolev KSE class $\mathfrak{A}_{\eta_{2},p}(Y,T)$, because the reconstruction process \eqref{eq:82} is invertible.
\end{proof}

\section{Sufficient Conditions for Well-posedness: Wide Applicability}\label{sec:5}
In this section, we give some useful sufficient conditions of well-posedness assumptions:
\begin{enumerate}[label=\textup{(A\arabic*)}]
	\item There exists a constant $C_0>0$ such that for all $x\in\mathbb{R}^n$,
	\begin{gather*}
		|f(x)|\le C_0(1+|x|^{2})^{p/2},\quad |\nabla f(x)|\le C_0(1+|x|^{2})^{(p-1)/2},\\
		|h(x)|\le C_0(1+|x|^{2})^{p/2},\quad |\nabla h(x)|\le C_0(1+|x|^{2})^{(p-1)/2},\\ |\Delta h(x)|\leq C_{0}(1+|x|^{2})^{(2p-1)/2}.
	\end{gather*}
	\item There exist constants $\beta_{\eta,p}>0$ and $C_{\eta,p}\ge0$ such that for all $x\in\mathbb{R}^{n}$,
	\begin{equation}
		\begin{aligned} 
			-\frac12 \operatorname{div} f(x) & -\frac12|h(x)|^2
			+\frac12 f(x)\cdot\nabla U_{\eta,p}(x)+\frac14|\nabla U_{\eta,p}(x)|^2+\frac14\Delta U_{\eta,p}(x) \\
			&\le C_{\eta,p}-\beta_{\eta,p}(1+|x|^{2})^{p}.
		\end{aligned} 
		\label{eq:86}
	\end{equation}
\end{enumerate}
on the coefficients of the filtering system \eqref{eq:model} for some fixed $\eta > 0$ and $p\geq 1$, with a view to showing that these assumptions are sufficiently general to cover most practically important systems, such as the detectable Kalman-Bucy filter. 

Generally speaking, the sufficient conditions presented below can be classified into three cases, namely:
\begin{itemize}
	\item the \textit{drift-dissipative case}, where well-posedness is mainly guaranteed by the dissipative nature of the drift term $f(x)$;
	\item the \textit{observation-dominated case}, where the observation function is sufficiently informative to ensure well-posedness;
	\item the \textit{drift-observation hybrid case}, where well-posedness stems from the combined effect of both the observation and the drift terms.
\end{itemize}

The first example is intended to illustrate that the well-posedness of the filtering problem can be analyzed within the weighted variational framework when the state dynamics in the filtering system \eqref{eq:model} is stable.

\begin{example}[Drift-dissipative cases]\label{ex:euclidean-linear}
	For a given $p\geq 1$, if the drift term $f(x)$ in the state equation of \eqref{eq:model} satisfies the dissipative condition:
	\begin{equation}
		x^{\top} f(x) \leq -c_{f}(1+|x|^{2})^{\frac{p+1}{2}} + C_{f},\ \forall \ x\in \mathbb{R}^{n},
		\label{eq:88}
	\end{equation}
	for some $c_{f}, C_{f} > 0$, then the condition \ref{A3RDMZ} holds for every observation function $h(x)$ satisfying \ref{A2}. In fact, the left-hand side of \eqref{eq:86} will be dominated by the term
	\begin{equation}
		\frac{1}{2}f(x)\nabla U_{\eta,p} (x) \leq -\eta (p+1)(1+|x|^{2})^{p} + \tilde{C}_{f},\ \forall\ x\in\mathbb{R}^{n},
	\end{equation}
	and the inequality \eqref{eq:86} holds accordingly. 
	
	Two important drift-dissipative cases are summarized as follows:
	
	Firstly, let $p=1$ and consider the affine drift term $f(x)=-Ax+a$ with $A+A^\top\ge2\alpha I_{n}$ for some $\alpha>0$. Let $h$ be any sensor satisfying \ref{A2} with $p=1$, for instance $h(x)=Hx+d$ with $H\in\R^{m\times n}$ arbitrary.  Then
	\[
	x^{\top} f(x) =-x^{\top}Ax+a^{\top}x
	\le -\alpha |x|^2+|a||x|,
	\]
	and \eqref{eq:88} holds for $p=1$. In this way, classical linear Kalman filtering systems with stable state and arbitrary linear observations can be studied using the robust DMZ equation theory.
	
	Secondly, for a given $p\geq 1$, if the drift term $f(x)$ can be written in the following form:	\[
	f(x)=-x(1+|x|^{2})^{\frac{p-1}{2}}+g(x),\qquad |g(x)|\le\kappa(1+|x|^{2})^{\frac{p}{2}},\qquad 0\le\kappa<1,
	\]
	with $|\nabla g(x)|\le C(1+|x|^{2})^{\frac{p-1}{2}}$, then
	\[
	\begin{aligned}
		x^{\top}f(x)
		&=-|x|^2(1+|x|^{2})^{\frac{p-1}{2}}+x^{\top}g(x)\leq -(1-\kappa)(1+|x|^{2})^{\frac{p+1}{2}}+(1+|x|^{2})^{\frac{p-1}{2}},
	\end{aligned}
	\]
	and we may choose $c_{f} = 1 - \kappa - \epsilon$ for some $0<\epsilon < 1-\kappa$, such that \eqref{eq:88} holds. 
\end{example}

Moreover, since the objective of the filtering problem is to produce effective estimates of an unknown dynamics based on observations, the drift term $f(x)$ often fails to satisfy the dissipative condition in many practical scenarios, and may lead to a chaotic or even unstable state process. The next example demonstrates that, as long as the observation function $h(x)$ is sufficiently informative, the filtering problem remains well-posed within our weighted variational framework, even when the drift $f(x)$ does not exhibit stability properties.

\begin{example}[Observation dominated cases]\label{ex:obs-coercive-sensors}
	Assume that the condition\ref{A2} holds for the drift term $f(x)$ and observation function $h(x)$, and for the same order $p$, the observation function $h(x)$ also satisfies
	\begin{equation} 
		|h(x)|\ge c_h (1+|x|^{2})^{p/2}-C_h, \ \forall\ x\in\mathbb{R}^{n},
		\label{eq:87}
	\end{equation}
	for some constants $c_{h}, C_{h} > 0$. 
	
	A direct computation shows that, under the growth assumption on $f$ in \ref{A2} and \eqref{eq:87}, the inequality \eqref{eq:86} holds for some constants $\beta_{\eta,p}, C_{\eta,p} > 0$. Consequently, the robust DMZ equation, the DMZ equation, and the Kushner-Stratonovich equation are all well-posed in this case.
	
	Moreover, since the observation function $h$ is continuous and \eqref{eq:87} holds on every compact set, the condition \eqref{eq:87} in fact imposes a lower bound on the growth rate of $h$ as $|x| \to \infty$. We therefore refer to this scenario as the observation-dominated case.
	
	Two classical yet important observation-dominated cases are worth recording here.
	
	First, let $p=1$ and let $h(x)=H_{1}x+H_{2}$ be affine, where the constant matrix $H_{1}\in\mathbb{R}^{m\times n}$ satisfies $H_{1}^{\top} H_{1}\ge cI_{m}$ for some constant $c > 0$. Then \eqref{eq:87} holds. Thus, full column-rank linear sensors and drift terms with linear or sublinear growth at infinity are among the observation dominated cases. Especially, the well-posedness of filtering equations for classical linear Kalman cases with full column rank observation matrix can be proved in the weighted variational setting of this paper. 
	
	Second, let $p\ge1$ and the observation function
	\[
	h(x)=(\gamma_1x_1^p,\ldots,\gamma_nx_n^p),\qquad \min_i|\gamma_i|>0,
	\]
	be the component-wise polynomial sensor. Then, \eqref{eq:87} holds for these $h(x)$ and as long as the growth rate of the drift term $f(x)$ satisfies \ref{A2}, the assumption \ref{A3RDMZ} will hold and the well-posedness of the filtering equations in this case has been derived. 
\end{example}

Finally, the last example concerns a hybrid case, in which the drift term $f(x)$ of the state process, though not dissipative, still possesses some degree of stability. In this scenario, the observation function $h(x)$ need not be as informative as in the previous observation-dominated case, yet the filtering problem remains well-posed within our weighted variational framework. A classical instance of this situation is the detectable Kalman-Bucy filter; we shall show that assumption \ref{A3RDMZ} largely corresponds to the detectability condition in the linear Kalman-Bucy setting.

\begin{example}[Drift-observation hybrid case]\label{ex:bounded}
	If neither the observation function is informative enough as in Example \ref{ex:obs-coercive-sensors} nor does the drift term satisfy the dissipative assumption in Example \ref{ex:euclidean-linear}, then the well-posedness of robust DMZ equation, DMZ equation and Kushner-Stratonovich equation may still be able to obtained in our weighted variational settings. The well-posedness stems from the interaction between $f(x)$ and $h(x)$ such that the condition \eqref{eq:86} holds.
	
	An important example for this hybrid case is the Kalman-Bucy filter with the detectability condition. 
	
	Let $f(x)=Fx+a$ and $h(x)=Hx+d$ with $p=1$.  Assume that $(F,H)$ is detectable, i.e., 
	there exists no eigenvalue $\lambda \in \sigma(F)$ with $\operatorname{Re}(\lambda) \ge 0$ and no nonzero vector $v \in \mathbb{C}^n$ (the corresponding right eigenvector) such that
	\begin{equation}
		F v = \lambda v, \quad H v = 0.
	\end{equation}
	Then, there exists a positive definite matrix $P\in\mathbb{R}^{n\times n}$, such that $FP + PF^{\top}$ is strictly negative definite in the null space of matrix $H$. 
	
	Without loss of generality, we may assume that $P = I_{n}$ is the identity matrix\footnote{For general positive definite matrix $P$, we may consider the invertible linear transformation $z_{t} = P^{\frac{1}{2}}x_{t}$.}. The left-hand side of \eqref{eq:86} becomes
	\begin{equation}
		C - \frac{1}{2}x^{\top}(H^{\top}H - 2\eta (F + F^{\top}) - 8\eta^2 I_{n} ) x,
	\end{equation}
	for some constant $C > 0$. The detectability condition implies \eqref{eq:86} holds for sufficient small $\eta > 0$. In fact, in the null space of the matrix $H$, we have $F+F^{\top}$ is strictly negative definite, while in the complement of the null space of $H$, the positive definite term $H^{\top} H$ dominates the quadratic form. 
	Therefore, the detectable Kalman-Bucy systems with rank-deficient observation matrix $H$ and stable unobserved modes is compatible with our weighted variational approach developed in this paper. 
\end{example}

The three cases above illustrate that the weighted variational approach proposed in this paper is applicable not only to systems with bounded coefficients, but also to those with unbounded coefficients; and not only to state equations with dissipative drift, but also to unstable dynamics with informative observations. It is worth emphasizing that this approach is, to a certain extent, specifically adapted to the filtering framework, where the observation plays a particularly significant role—beyond that in the state equation alone. This point is further elucidated by the following counterexample, which demonstrates that unstable dynamics with uninformative observations may fall outside the scope of the present framework.

\begin{example}[Counterexample: unstable dynamics without observations]
	Let us consider the one-dimensional linear stochastic differential equation:
	\begin{equation}
		dx_{t} = a x_{t} dt + dv_{t},\ x_{0}\sim \mathcal{N}(0,1),\ t\in[0,T],
		\label{eq:92}
	\end{equation}
	where $a > 0$ is a constant. The linear dynamics \eqref{eq:92} is unstable in the sense that the variance of $x_{t}$, which is explicitly given by
	\begin{equation}
		\text{Var} (x_{t}) = e^{2at}\left(1+\frac{1}{2a}\right) - \frac{1}{2a},\ \forall\ t\in [0,T],
	\end{equation}
	will tend to infinity with an exponential rate as $t\rightarrow\infty$. 
	
	If there are no observations, i.e., $h(x)\equiv 0$ in \eqref{eq:model}, then the robust DMZ equation \eqref{eq:rdmz}, the DMZ equation \eqref{eq:dmz} as well as the Kushner-Stratonovich equation \eqref{eq:kse} will all reduce to the Fokker-Planck equation corresponding to \eqref{eq:92}, that is,
	\begin{equation}
		\frac{\partial}{\partial t}p(t,x) = \frac{1}{2}\frac{\partial^{2}}{\partial x^{2}} p(t,x) - a\frac{\partial }{\partial x}(xp(t,x)),\ t\in[0,T],
		\label{eq:94}
	\end{equation}
	with initial value $p(0,x) = \frac{1}{\sqrt{2\pi}} e^{-\frac{1}{2}x^{2}}$. The solution of equation \eqref{eq:94} has an explicit form, which is given by 
	\begin{equation}
		p(t,x) =\sqrt{\frac{a}{\pi(e^{2at}(2a+1) - 1)}}\exp\left(-\frac{ax^{2}}{e^{2at}(2a+1) - 1}\right),\ t\in[0,T]. 
	\end{equation}
	A necessary and sufficient condition for $p(t,\cdot) \in H_{\eta,p}$ with $p=1$ is given by
	\begin{equation}
		\eta < \frac{a}{e^{2at}(2a+1) - 1}.
		\label{eq:96}
	\end{equation}
	Notice that the right-hand side of \eqref{eq:96} decays exponentially in time $t$; consequently, the parameter $\eta$ must be chosen in accordance with the terminal time $T$. In the meanwhile, the left-hand side of \eqref{eq:86} becomes
	\begin{equation}
		-a-\eta + (2a\eta + 4\eta^2) x^{2},
	\end{equation}
	and thus, the condition \ref{A3RDMZ} does not hold in this case. 
\end{example}

In summary, the three examples and the counterexample presented in this section indicate that the weighted variational framework developed in this paper is well aligned with the practical demands of filtering theory, namely, the efficient tracking of stochastic dynamics based on informative observations.

\section{Conclusion}\label{sec:conclusion}
%=====================================================================
In this paper, we present a weighted variational framework to study the well-posedness of important evolution equations in nonlinear filtering theory, i.e., the robust DMZ equation, the stochastic DMZ equation and the Kushner-Stratonovich equation. The existence and uniqueness of a weak solution to the robust DMZ equation is first derived, and the stochastic DMZ equation and Kushner-Stratonovich equation are then studied in the buffered weighted Sobolev spaces. 

The weighted Sobolev spaces, together with their buffered versions, are compatible with the gauge (exponential) transformation linking the stochastic DMZ equation and the robust DMZ equation. Instead of introducing weight functions of a new form, it suffices to adjust a single parameter in the weight function when passing from the study of solutions to the robust DMZ equation to that of the stochastic DMZ equation and the Kushner-Stratonovich equation. Consequently, under fairly general conditions—allowing, in particular, unbounded coefficients in the filtering system—we are able, for the first time, to provide a unified treatment of all three equations within a family of buffered spaces.

Moreover, sufficient conditions for the well-posedness of the filtering equations are discussed in this paper, illustrating the wide applicability of the weighted variational framework, including the important case of detectable Kalman-Bucy filters. These sufficient conditions also demonstrate that the weighted variational framework introduced herein is naturally aligned with the fundamental objective of the filtering problem—namely, to estimate the trajectory of an unknown dynamics by means of informative observations.

Some promising research directions are outlined as follows. Firstly, the energy estimates in this paper are carried out for the robust DMZ equation; a direct energy estimate for the DMZ equation and the Kushner-Stratonovich equation within the weighted variational framework remains to be developed, which would provide an alternative probabilistic perspective on the well-posedness of the filtering equations. Secondly, the weighted variational framework introduced herein is employed to study filtering equations on a finite time horizon. The long-time behavior of these equations is also a subject of significant interest in both theory and industrial practice. Finally, the weighted variational framework may be extended to more general settings, including filtering systems with time-dependent coefficients, correlated state and observation noises, and more general L\'evy-type noise.

%=====================================================================

%=====================================================================

\appendix

\section{Lions–Magenes Theorem}\label{app:lions}
%=====================================================================

This appendix recalls the Lions–Magenes theorem, which has been used in the proof of Theorem \ref{thm:rdmz}. For the reader's convenience, we state and prove here a version of this theorem that is better adapted to the notation and assumptions used in the main text. For a historical account, the reader is referred to Lions and Magenes \cite[Ch.~3, Sect.~1 and 4]{LionsMagenes72} and Showalter \cite[Ch.~III]{Showalter97}.

\begin{theorem}[Lions–Magenes]\label{thm:lions}
	Let \(V\hookrightarrow H\cong H'\hookrightarrow V'\) be a Gelfand triple of
	real separable Hilbert spaces whose first embedding is dense and continuous. For \(t\in[0,T]\), let
	\(a(t;\cdot,\cdot):V\times V\to\R\) be a bilinear form such that
	\begin{enumerate}[label=\textup{(\alph*)}]
		\item \(t\mapsto a(t;v,w)\) is Lebesgue measurable for every \(v,w\in V\);
		\item there is \(M>0\) such that
		\[
		|a(t;v,w)|\le M\norm{v}_V\norm{w}_V
		\quad\text{for a.e. }t\text{ and all }v,w\in V;
		\]
		\item there are \(\alpha>0\) and \(\lambda\ge0\) such that
		\[
		a(t;v,v)\ge
		\alpha\norm{v}_V^2-\lambda\norm{v}_H^2
		\quad\text{for a.e. }t\text{ and all }v\in V.
		\]
	\end{enumerate}
	Then, for every \(u_0\in H\) and
	\(F\in L^2(0,T;V')\), there exists a unique
	\[
	u\in L^2(0,T;V)\cap C([0,T];H),
	\ \text{with} \ \partial_{t} u\in L^2(0,T;V'),
	\]
	such that \(u(0)=u_0\) and
	\begin{equation}\label{eq:A-lions-equation}
		\dual{\partial_t u(t)}{v}_{V',V}+a(t;u(t),v)
		=\dual{F(t)}{v}_{V',V}
		\quad\text{for a.e. }t\in(0,T)\text{ and all }v\in V.
	\end{equation}
\end{theorem}

\begin{proof}
	Firstly, we may reduce the theorem to the coercive case. Set
	\[
	u(t)=e^{\lambda t}z(t).
	\]
	A direct calculation in \(V'\) shows that the element $u$ satisfies \eqref{eq:A-lions-equation}
	is equivalent to
	\begin{equation}\label{eq:A-shifted-equation}
		\dual{\partial_t z(t)}{v}_{V',V}+b(t;u(t),v)
		=\dual{e^{-\lambda t}F(t)}{v}_{V',V}
		\quad\text{for a.e. }t\in(0,T)\text{ and all }v\in V.
	\end{equation}
	where the shifted form defined by
	$b(t;v,w):=a(t;v,w)+\lambda\ip{v}{w}_H$ is bounded and coercive:
	\begin{equation}\label{eq:A-coercive}
		b(t;v,v)\ge\alpha\norm{v}_V^2.
	\end{equation}
	It is therefore enough to prove the theorem in the coercive case
	\(\lambda=0\).  
	
	Since the embedding $V\hookrightarrow H$ is continuous and dense, there exists a sequence \(\{u_0^N\}_{N=1}^{\infty}\subset V\) satisfying both
	\begin{equation}\label{eq:A-initial-approx}
		u_0^N\longrightarrow u_0\quad\text{strongly in }H,
		\ \text{and} \ 
		\frac{1}{N}\norm{u_0^N}_V^2\longrightarrow0. 
	\end{equation}
	Fix \(N\in\mathbb N\), and set
	$
	\tau:=\frac{T}{N},\ 
	t_k:=k\tau, \ k=1,\cdots,N. 
	$
	
	Starting from \(u_0^N\), we may define a time-discretization solution \(u_k^N\in V\), \(k=1,\ldots,N\) of \eqref{eq:A-lions-equation}, successively
	by
	\begin{equation}\label{eq:A-Rothe}
		\ip{\frac{u_k^N-u_{k-1}^N}{\tau}}{v}_H
		+a(t_{k};u_k^N,v)
		=\dual{F(t_{k})}{v}_{V',V},
		\qquad v\in V.
	\end{equation}
	Indeed, after multiplication by \(\tau\), \eqref{eq:A-Rothe} becomes:
	\begin{equation}
		\ip{u_{k}^{N}}{v}_{H} + \tau a(t_{k};u_{k}^{N},v) = \ip{u_{k-1}^{N}}{v}_{H} + \tau\dual{F(t_{k})}{v}_{V',V}.
		\label{eq:103}
	\end{equation}
	The bilinear form on the left-hand side of \eqref{eq:103} is coercive and the right-hand side is a bounded functional of $v$. Thus, the Lax-Milgram theorem gives a unique
	\(u_k^N\in V\) at each step.  This constructs the entire discrete trajectory.
	
	Choosing \(v=u_k^N\) in \eqref{eq:A-Rothe}, multiplying by \(2\tau\), and using
	\begin{equation}\label{eq:A-polarization-discrete}
		2\ip{u_k^N-u_{k-1}^N}{u_k^N}_H
		=
		\norm{u_k^N}_H^2-\norm{u_{k-1}^N}_H^2
		+\norm{u_k^N-u_{k-1}^N}_H^2,
	\end{equation}
	we deduce the discrete energy inequality:
	\begin{equation} 
		\begin{aligned}
			&\|u_k^N\|_H^2-\|u_{k-1}^N\|_H^2
			+\|u_k^N-u_{k-1}^N\|_H^2
			+\alpha\tau\|u_k^N\|_V^2
			\leq
			\frac{\tau}{\alpha}\|F(t_{k})\|_{V'}^2,\ k=1,\cdots,N. 
		\end{aligned}
		\label{eq:105}
	\end{equation} 
	Taking the summation of \eqref{eq:105} over $k$ from 1 to $N$, we obtain:
	\begin{equation}\label{eq:A-discrete-bounds}
		\max_{0\le k\le N}\norm{u_k^N}_H^2 + 
		\sum_{k=1}^N\norm{u_k^N-u_{k-1}^N}_H^2 + \alpha\tau\sum_{k=1}^{N}\|u_k^N\|_V^2\le \norm{u_0^N}_H^2
		+\frac1\alpha\norm{F}_{L^2(0,T;V')}^2,
	\end{equation}
	The right-hand side of \eqref{eq:A-discrete-bounds} is bounded because \(u_0^N\to u_0\) in \(H\).
	
	Regarding the $\frac{u_k^N-u_{k-1}^N}{\tau}$ in \eqref{eq:A-Rothe} as an element in $V'$, we then obtain
	\begin{align}
		\tau\sum_{k=1}^N\norm{\frac{u_k^N-u_{k-1}^N}{\tau}}_{V'}^2
		&\le
		2\tau\sum_{k=1}^N\norm{F(t_{k})}_{V'}^2
		+2M^2\tau\sum_{k=1}^N\norm{u_k^N}_V^2 \\
		&\le
		2\norm{F}_{L^2(0,T;V')}^2+\frac{2M^2}{\alpha}\left(\norm{u_0^N}_H^2
		+\frac1\alpha\norm{F}_{L^2(0,T;V')}^2,\right)
		\label{eq:A-dt-discrete}
	\end{align}
	Thus both the discrete solutions and their discrete derivatives have bounds
	independent of \(N\).
	
	Define the right-continuous piecewise constant functions
	\[
	\bar u_N(t):=u_k^N,\qquad
	\bar F_N(t):=F(t_{k}),\ t\in (t_{k-1}, t_{k}], \ k=1,\cdots,N,
	\]
	and the continuous piecewise affine interpolant
	\begin{equation}\label{eq:A-affine-interpolant}
		\widehat u_N(t)
		:=
		u_{k-1}^N+(t-t_{k-1})\frac{u_{k}^{N} - u_{k-1}^{N}}{\tau},
		\qquad t\in[t_{k-1},t_k].
	\end{equation}
	Then \(\partial_{t} \widehat u_N=\frac{u_{k}^{N} - u_{k-1}^{N}}{\tau}\) a.e. on \((t_{k-1},t_{k}]\), and 
	\[
	\widehat u_N(t)-\bar u_N(t)
	=-(t_{k} - t)\frac{u_{k}^{N} - u_{k-1}^{N}}{\tau}. 
	\]
	Consequently,
	\begin{equation}\label{eq:A-interpolant-difference}
		\begin{aligned} 
		\norm{\widehat u_N-\bar u_N}_{L^2(0,T;V')}^2
		&=
		\frac{\tau^3}{3}\sum_{k=1}^N\norm{\frac{u_{k}^{N} - u_{k-1}^{N}}{\tau}}_{V'}^2 \\
		&=
		\frac{\tau^2}{3}
		\left(\tau\sum_{k=1}^N\norm{\frac{u_{k}^{N} - u_{k-1}^{N}}{\tau}}_{V'}^2\right)
		\xrightarrow{N\rightarrow\infty}0.
		\end{aligned} 
	\end{equation}
	The above boundedness results imply that there exists an element \(u\in L^2(0,T;V)\), with \(\partial_{t} u\in L^2(0,T;V')\), and \(\bar u\in L^2(0,T;V)\), such that
	\begin{equation}\label{eq:A-Rothe-convergence}
		\widehat u_N\rightharpoonup u
		\quad\text{in }\mathcal L^{2}(0,T; V),\  \partial_t \widehat u_N\rightharpoonup  \partial_t u
		\ \text{in }\mathcal L^{2}(0,T; V'),
		\ 
		\bar u_N\rightharpoonup\bar u
		\quad\text{in }L^2(0,T;V).
	\end{equation}
	and the strong convergence
	\eqref{eq:A-interpolant-difference} forces \(\bar u=u\).
	
	Notice that the time-discretization equality \eqref{eq:A-Rothe} can be written as
	\begin{equation}
		\ip{\partial_{t}\hat{u}_{N}}{v}_H
		+a(t_{k};\bar{u}_k^N,v)
		=\dual{F(t_{k})}{v}_{V',V},
		\qquad v\in V.
		\label{eq:112}
	\end{equation}
	For \(\phi\in L^2(0,T;V)\), integrating \eqref{eq:112} with respect to $t$ over $[0,T]$ gives
	\begin{equation}\label{eq:A-Rothe-weak}
		\int_0^T\dual{\partial_{t} \widehat u_N}{\phi}_{V',V}\,dt
		+\int_0^T a(t;\bar u_N,\phi)\,dt
		=
		\int_0^T\dual{\bar F_N}{\phi}_{V',V}\,dt.
	\end{equation}
	
	Letting \(N\to\infty\) in \eqref{eq:A-Rothe-weak}, we find
	\[
	\int_0^T\dual{\partial_{t}u}{\phi}\,dt
	+\int_0^T a(t;u,\phi)\,dt
	=
	\int_0^T\dual{F}{\phi}\,dt
	\qquad\text{for every }\phi\in L^2(0,T;V).
	\]
	Since $\phi$ is arbitrarily chosen, we have
	\begin{equation}
		\dual{\partial_t u(t)}{v}_{V',V}+a(t;u(t),v)
		=\dual{F(t)}{v}_{V',V}
		\quad\text{for a.e. }t\in(0,T)\text{ and all }v\in V,
	\end{equation}
	and $u$, together with $\partial_{t} u$, is a weak solution to \eqref{eq:A-lions-equation}. 
	
	For the uniqueness result, let \(u_1,u_2\) be two solutions to \eqref{eq:A-lions-equation} and put \(w=u_1-u_2\).  Then
	\(w\in L^{2}(0,T;V)\) with \(\partial_{t} w\in L^{2}(0,T;V')\), \(w(0)=0\), and satisfies
	\begin{equation}
		\dual{\partial_t w(t)}{v}_{V',V}+a(t;w(t),v)
		=0
		\quad\text{for a.e. }t\in(0,T)\text{ and all }v\in V,
		\label{eq:115}
	\end{equation}
	Take $v = w$ in \eqref{eq:115}, and we obtain 
	\[
	\frac12\frac{d}{dt}\norm{w(t)}_H^2
	=-a(t;w(t),w(t))
	\le\lambda\norm{w(t)}_H^2
	\quad\text{a.e.}
	\]
	Gr\"onwall's inequality implies \(w=0\), which proves the uniqueness.
\end{proof}

\bibliographystyle{siamplain}
\bibliography{references}

\newpage

\section*{Technical Report on AI-Assisted Research}
This document serves as a technical report for the AI-assisted research conducted during this study. During the course of this research, AI Mathematician (AIM)\footnote{Y. Liu, Y. Huang, Y. Wang, P. Li, and Y. Liu, AI Mathematician: Towards Fully Automated Frontier Mathematical Research, arXiv preprint arXiv:2505.22451, 2025.}, an AI research agent, is acknowledged for contributing to some of the proof ideas and early-stage proof strategies, and also generating the draft proofs for human verification.. This technical report provides a systematic documentation of the AI-assisted contributions and is based on the output and log files of AIM.

\subsection{Scope and division of labor}

The human authors supplied the filtering model, the objective of proving well-posedness for the robust Duncan-Mortensen-Zakai, stochastic DMZ, and Kushner-Stratonovich equations with polynomially growing coefficients, and the standards by which a proposed proof would be accepted. They also selected the final hypotheses, examples, theorem statements, and exposition. AIM was used as an exploratory mathematical assistant: it compared candidate weighted spaces, expanded energy identities, organized the variational proof into lemmas, and subjected intermediate claims to repeated adversarial review. Every result retained in the manuscript was selected and checked by the human authors.

\subsection{The principal contribution: selecting the weight}
The most consequential AIM contribution was to turn the vague instruction ``choose a weight that absorbs the unbounded coefficients'' into a precise choice of exponent. In the notation of Section~\ref{sec:framework}, the weight is
\[
w_{\eta,p}(x)=e^{U_{\eta,p}(x)},\ \text{with} \ \ 
U_{\eta,p}(x)=2\eta\,(1+|x|^2)^{\frac{p+1}{2}},
\]
so that the spaces $H_{\eta,p}$ and $V_{\eta,p}$ in \eqref{eq:wl2} and \eqref{eq:spaces} are weighted by a single exponential with buffering parameter $\eta > 0$ and regularity parameter $p\geq 1$, which is decided by the growth order of the coefficients in assumption \ref{A2}. 

Early AIM paths fixed a Gaussian weight $e^{-\alpha|x|^2}$, a quadratic exponent that is adequate only in a comparatively strong observation-coercivity regime. A separate super-exponentially decaying candidate $e^{-\kappa|x|^{2p+2}}$ was also rejected in the review regime of AIM, because, as in the output of AIM, it overwhelms the negative potential of order $|x|^{2p}$ available from the observation term $-\tfrac12|h|^2$ in \eqref{eq:A3RDMZ-R0}. However, this diagnosis isolated the correct scale $(1+|x|^2)^{(p+1)/2}$.

In the meanwhile, the inverse-gauge factor $e^{Y_t^\top h(x)}$ in \eqref{eq:gauge} has exponent of order $|x|^p$, so the weight must grow faster, which results in the buffered-space estimate \eqref{eq:54b}-\eqref{eq:55}. Based on the above weight functions, the following assumption
\[
-\tfrac12\divg f-\tfrac12|h|^2+\tfrac12 f\cdot\nabla U_{\eta,p}+\tfrac14|\nabla U_{\eta,p}|^2+\tfrac14\Delta U_{\eta,p} \leq C_{\eta,p} - \beta_{\eta,p}(1+|x|^{2})^{p},
\]
i.e.\ exactly \eqref{eq:A3RDMZ-R0}, was proposed. The endpoint choice of the weight function and this final assumption were then fixed in the authors' revision instructions; AIM's defensible contribution is the earlier identification and mathematical testing of the $(1+|x|^2)^{(p+1)/2}$ scale, followed by the systematic implementation of the adopted assumptions.

\subsection{From the weight to a unified proof architecture}

The iterative workflow in AIM also helped assemble the weighted Gelfand triple $(V_{\eta,p},H_{\eta,p},V_{\eta,p}')$ of Lemma~\ref{lem:2.1} and the correct bilinear form $a_{Y,\eta,p}$ in \eqref{eq:form}. 
The same exponent $(1+|x|^2)^{(p+1)/2}$ also produces the buffered-space mechanism of Section~\ref{sec:4}. For $0<\eta'<\eta$, the margin $(\eta-\eta')(1+|x|^2)^{(p+1)/2}$ dominates the factor $Y_t\cdot h(x)$ of order $|x|^p$; consequently the inverse gauge $\sigma=e^{Y_t^\top h(x)}u$ of \eqref{eq:gauge} maps a solution $u$ in the $\eta$-weighted space $H_{\eta,p}$ into the weaker $\eta'$-weighted space $H_{\eta',p}$, without changing the form of the weight, as shown by the estimate \eqref{eq:54b} in the proof of Theorem~\ref{thm:zakai}. This parameter loss from $\eta$ to $\eta'$ is precisely what allows the robust DMZ equation \eqref{eq:rdmz}, the stochastic DMZ equation \eqref{eq:dmz}, and the normalized Kushner-Stratonovich equation \eqref{eq:kse} to be treated within one family of buffered weighted spaces, culminating in the well-posedness results of Theorem~\ref{thm:rdmz} and Theorems~\ref{thm:zakai}-\ref{thm:kse}.

\subsection{Verification and limitations}

The record includes unsuccessful branches. Human authors' review found that an early AIM proof required an unstated coercivity inequality, did not match the requested function space, and left bounded-domain compatibility and uniqueness gaps. Later expansion attempts also remained partial. These failures were not counted as mathematical results; they helped determine the final hypotheses and proof route. The surviving attribution is supported by dated local proof paths, AIM session records, and successive revision files. Classical gauge transformations, the Lions-Magenes theorem, the final coefficient condition, and the authors' mathematical verification are not attributed to AIM.

\end{document}